\documentclass[11pt,a4paper]{article}
\usepackage{amsmath,amssymb,amsthm}
\usepackage[margin=25mm]{geometry}
\usepackage{longtable}
\usepackage{authblk}
\newtheorem{theorem}{Theorem}[section]
\newtheorem{lemma}[theorem]{Lemma}
\newtheorem{proposition}[theorem]{Proposition}
\newtheorem{corollary}[theorem]{Corollary}
\theoremstyle{definition}\newtheorem{definition}[theorem]{Definition}
\theoremstyle{remark}
\newcommand{\zrl}{z_{\mathrm{RL}}}
\newcommand{\orth}{\perp_R}
\newcommand{\RW}{(\mathrm{RW3}^{+})}

\newcommand{\dd}{\delta}

\title{A Twelve-Row Seed and a One-Row Extension for\\Five-Column Recursive-Line Zarankiewicz Numbers}
\author[1]{Min Xi}
\author[2]{Jingya Chang\thanks{Corresponding author: jychang@gdut.edu.com}}

\affil[1]{School of Mathematics and Statistics, Guangdong University of Foreign Studies, Guangzhou, China}
\affil[2]{School of Mathematics and Statistics,  Guangdong University of Technology, Guangzhou, China}

\date{\today}
\begin{document}\maketitle

\begin{abstract}
We prove that the five-column recursive-line Zarankiewicz number satisfies
$
z_{RL}(m,5)=3m+5
$
for every integer \(m\ge 12\). The proof is based on an explicit \(12\times5\) seed configuration combined with a recursive one-row extension scheme. The seed attains the five-column cell bound and contains no unoccupied cells. Two selected pairs in the seed are opened and replaced by parallel paths, ensuring that each inserted row increases the total number of augmented edges by three while preserving the required simple configuration.
The main challenge lies in verifying the strengthened recursive-line condition \({\rm (RW3+)}\) uniformly across all extension lengths. To this end, we develop a distance-reducing rectangle lemma that transfers certified inner-product relations along extension paths, reducing relations between distant labels to those at smaller path distances. This confines the verification for arbitrarily long extensions to a finite collection of seed and interface certificates. The resulting construction successfully satisfies pair identification, preserves distinct selected-edge classes, and certifies the orthogonality of distinct edge representatives.
Furthermore, combining this constructed lower bound with the parameter hierarchy yields
$
z_2(m,5)=z_{RL}(m,5)=3m+5
$
throughout this range.
\end{abstract}


\section{Introduction}

The classical Zarankiewicz problem asks for the maximum number of occupied
cells in an $m\times n$ array subject to the condition that no $2\times2$
subarray is completely occupied. Equivalently, it asks for the maximum
number of edges in a $C_4$-free bipartite graph with prescribed part sizes.
The problem was introduced by Zarankiewicz \cite{Z1951}, and its
basic extremal estimates and structural developments have been studied
extensively in extremal graph theory; see, for example, the work of
K\H{o}v\'ari, S\'os and Tur\'an \cite{KST1954}, Cul\'ik \cite{Culik1956},
Reiman \cite{Reiman1958}, and the broader accounts in
\cite{Guy1969,Furedi1996,FurediSimonovits2013}.  For fixed small numbers of
columns, the problem becomes considerably more rigid and exact values are
available in a number of cases.  In particular, for five columns one has
\[
    z(m,5)=m+10,\qquad m\geq 10,
\]
as recorded in \cite{GZ1969}.  This classical value provides the
one-edge extremal boundary from which the denser recursive-line problem
considered here is developed.

The relevance of the Zarankiewicz problem to the present setting comes from
its connection with sums of squares of biquadratic forms.  A biquadratic
form can be represented by squares of bilinear forms, and the minimum number
of squares in such a representation is its SOS rank.  The general theory of
sums of squares provides the algebraic background for this viewpoint
\cite{Choi1975,ChoiLamReznick1995}.  In the combinatorial representation
used in the Zarankiewicz setting, a $C_4$-free bipartite graph supplies the
basic one-edge configuration, while additional algebraic structure can be
encoded by allowing selected pairs of cells to represent squares of
two-term bilinear forms.  This observation led to a sequence of
combinatorial extensions of the classical parameter, in which the
combinatorial configuration is required to preserve irreducibility of the
associated biquadratic form.

The first such extension relevant here is the augmented Zarankiewicz
framework.  Qi, Cui, and Xu introduced the augmented Zarankiewicz number and
the limited augmented Zarankiewicz number, denoted by $z_A(m,n)$ and
$z_L(m,n)$, respectively \cite{QiCuiXu2026Rank,QiCuiXu2026Augmented}.
The limited version fixes the one-edge part at the classical extremal value
$|E_1|=z(m,n)$ while allowing selected two-edges.  This gives a direct
combinatorial mechanism for obtaining SOS-rank lower bounds beyond the
classical Zarankiewicz number, summarized by
\[
    \operatorname{BSR}(m,n)
    \geq z_A(m,n)
    \geq z_L(m,n)
    \geq z(m,n).
\]
Thus the role of the classical Zarankiewicz number in the present problem
is not discarded; rather, it becomes the fixed one-edge skeleton on which
the augmented configurations are built.

The augmented admissibility conditions, however, are stronger than is
necessary for irreducibility.  Qi, Cui, and Xu subsequently introduced a
weak version of the framework, leading to the weak limited augmented
Zarankiewicz number $z_{wL}(m,n)$.  The weak formulation relaxes the
original local restrictions while retaining sufficient conditions for the
associated doubly simple biquadratic form to remain irreducible
\cite{QiCuiXu2026Weak}.  In particular, the resulting parameters satisfy
\[
    \operatorname{BSR}(m,n)
    \geq z_{wL}(m,n)
    \geq z_L(m,n)
    \geq z(m,n).
\]
This relaxation is important because it shows that configurations excluded
by the original augmented conditions can nevertheless contribute to
irreducible SOS representations.  The weak framework therefore enlarges
the class of admissible configurations, but it still relies on a
cross-cell condition to certify the required orthogonality relations.

The next step is to replace that literal local condition by the recursive
algebraic mechanism that is actually used by the orthogonality relations.
L\"ofberg and Qi introduced the second order Zarankiewicz number $z_2(m,n)$
and the intermediate recursive-line parameter $z_{RL}(m,n)$
\cite{LQ}.  The recursive-line framework retains the simplicity
and $C_4$-free requirements of the underlying configuration but replaces
the literal weak cross-cell test by a recursive rectangle condition,
strengthened to $(RW3^+)$.  The resulting parameters fit into the hierarchy
\[
    \operatorname{BSR}(m,n)
    \geq z_2(m,n)
    \geq z_{RL}(m,n)
    \geq z_{wL}(m,n)
    \geq z(m,n).
\]
The significance of this step is that orthogonality need not always be
certified by an immediately visible unoccupied opposite cell.  It may
instead be propagated through a finite sequence of rectangle identities.
Consequently, the recursive-line parameter provides a natural setting in
which configurations denser than those admitted by the weak framework can
still be certified as irreducible.  Recent results already demonstrate this
separation in small dimensions and in the three-column case
\cite{QiLofbergChen2026}.

We focus on the five-column case. Combining the classical value $z(m,5)=m+10$ with the cell-conservation identity
\[
|E_1|+2|E_2|+H=5m,
\]
where $H$ denotes the number of unoccupied cells, gives the general density bound
\[
z_{RL}(m,5)\leq z_2(m,5)\leq3m+5.
\]

Attaining the maximum cell-density ceiling $3m+5$ for an arbitrary row count $m$ presents a double structural challenge. Geometrically, any candidate configuration must be strictly hole-free ($H=0$), leaving no unused cell capacity. More critically, the primary obstacle lies in satisfying the recursive admissibility under $\text{RW3}^+$.

Achieving the correct edge counts is not enough, because selected two-edges introduce strict algebraic constraints. Under $\text{RW3}^+$, orthogonality must hold globally among all edge representatives. Because rectangle identities can propagate constraints across different parts of the grid, a locally valid extension can easily violate global consistency. Consequently, as the row dimension $m$ grows, relations spanning large distances cannot be verified through local checks or finite computer searches, creating a major obstacle for arbitrary $m \ge 12$.

To overcome this obstacle, our approach separates the finite base construction from the infinite extension argument. We begin with an explicit $12 \times 5$ seed configuration that achieves the exact density bound $z_{RL}(12,5) = 3 \cdot 12 + 5$. This seed is chosen not only to attain the cell bound but also to incorporate a compatible interface structure with selected pairs. These pairs are opened and replaced by a dual-path extension: each inserted row introduces one new one-edge and increases the number of two-edges by two in net, so that the total number of augmented edges increases by three while keeping the grid strictly hole-free.

To certify $\text{RW3}^+$ across arbitrary row counts, we establish a distance-reducing path lemma. This lemma reduces inner-product defects between distantly separated labels to smaller separations, ultimately reaching certified relations within the seed or interfaces. This provides a finite control mechanism for infinite extensions, avoiding the need to verify individual values of $m$.

This construction yields the following result.
\[
\boxed{z_{RL}(m,5)=3m+5,\qquad m\geq12.}
\]
Since
\[
z_{RL}(m,5)\leq z_2(m,5)\leq3m+5,
\]
the same construction also gives
\[
z_2(m,5)=z_{RL}(m,5)=3m+5,\qquad m\geq12.
\]
Thus the five-column recursive-line problem attains its general cell-density ceiling for every $m\geq12$, providing an explicit infinite family in which the recursive-line lower bound matches the upper bound.

The remainder of the paper is organized as follows. Section~2 recalls the augmented-grid terminology, the certificate relations, the rectangle closure rules, and the strengthened recursive-line condition $(\text{RW3}^+)$. Section~3 constructs the $12\times5$ seed and describes the recursive double-path extension, together with the resulting edge counts and simplicity properties. Section~4 establishes the uniform resolution of the selected pairs, proves the distance-reducing path lemma, and derives the orthogonality relations within the extended part. Section~5 proves compatibility between the extended paths and the fixed seed and completes the verification of $(\text{RW3}^+)$ for every extension length. The appendices record the finite certificate paths used in the seed and interface verifications.

\section{Preliminaries and Closure Framework}
\label{sec:preliminaries}

We work throughout with the strengthened recursive-line closure framework for augmented bipartite graphs introduced by L\"ofberg and Qi~\cite{LQ}.
The purpose of this section is to fix the terminology and notation used in the construction and to recall the certificate rules underlying the admissibility condition.
We first describe the augmented grid configuration and its cell accounting, then introduce the certificate relations and rectangle closure rules, and finally state the strengthened recursive-line condition used throughout the paper.

\subsection{Augmented Grid Configurations}
\label{sec:augmented_grid}

We represent an augmented bipartite graph on an $m\times 5$ grid whose cells are indexed by

$$
p=(r,c),\qquad 1\le r\le m,\quad 1\le c\le 5.
$$

An occupied cell contains either a one-edge or one half of a selected two-edge.
A cell containing no edge element is called \emph{unoccupied} (or a \emph{hole}).
Let $H$ denote the total number of unoccupied cells.

Let $E_1$ denote the set of one-edges and let $E_2$ denote the set of selected two-edges. Each selected two-edge consists of two occupied cells, and the supports of distinct selected two-edges are required to be disjoint. The underlying one-edge graph must be $C_4$-free. A configuration satisfying these conditions is called \emph{simple}.

Since every one-edge occupies one cell and every selected two-edge occupies two cells, counting the occupied and unoccupied cells gives the basic cell-conservation identity
\begin{equation}
|E_1|+2|E_2|+H=5m.
\label{eq:cell_conservation}
\end{equation}

The notation $E_1$, $E_2$, and $H$ will be used throughout the construction.
In particular, the equality case $H=0$ corresponds to a hole-free configuration, which is the form used by the recursive construction below.

\subsection{Certificate Relations and Rectangle Closure}
\label{sec:certificate_closure}

The recursive-line framework does not require the relevant orthogonality relations to be checked by direct numerical inner-product computation.
Instead, relations are certified recursively from a small collection of geometric rules.

For every occupied cell $p$, associate a formal unit coefficient vector $v_p$.
For two distinct occupied cells $p$ and $q$, define

$$
\delta(p,q)=
\begin{cases}
1, & \text{if } \{p,q\}\in E_2,\\
0, & \text{otherwise}.
\end{cases}
$$

Thus, the value $\delta(p,q)=1$ records that the two cells form the two halves of the same selected two-edge.

\paragraph{Line rule.}
If two occupied cells $p$ and $q$ lie in a common row or a common column, their coefficient vectors satisfy
\begin{equation}
\langle v_p,v_q\rangle=\delta(p,q).
\label{eq:line_rule}
\end{equation}
Consequently, a selected pair lying on a common row or column is immediately certified as a pair of identical representatives, while two distinct edge elements on a common row or column are certified to be orthogonal.

\paragraph{Rectangle rule.}
Consider a genuine $2\times2$ rectangle in the grid.
Let ${p,q}$ and ${r,s}$ be its two diagonals.
The coefficient vectors satisfy the rectangle identity
\begin{equation}
\langle v_p,v_q\rangle+\langle v_r,v_s\rangle
=\delta(p,q)+\delta(r,s).
\label{eq:rectangle}
\end{equation}
Hence, whenever the value of one diagonal is already certified, the value of the complementary diagonal is determined by the rectangle identity.

In particular, if a diagonal contains an unoccupied cell, its inner product is zero.
The rectangle rule can therefore transfer a certified zero or one relation from one diagonal to the other.

\paragraph{Complementary-pair rule.}
A special case occurs when both diagonals of a genuine rectangle are selected two-edges.
The rectangle identity then certifies the corresponding pair identifications simultaneously.
This rule is useful in configurations in which two selected pairs interact through the same rectangle.

The preceding rules generate a finite system of certified relations.
We write
$
p\sim q
$
when the closure rules certify
$
\langle v_p,v_q\rangle=1,
$
so that the two cell vectors represent the same edge element.
Likewise, we write
$
p\perp_R q
$
when the closure rules certify
$
\langle v_p,v_q\rangle=0.
$

The subscript $R$ emphasizes that the relation is obtained through the recursive rectangle-closure procedure rather than by an independent numerical calculation.

More generally, after selected pairs have been identified, the same notation is used for the corresponding edge representatives.
Thus, an assertion that two distinct edge elements are \emph{certified orthogonal} means that their orthogonality has been obtained from the line and rectangle closure rules.

The closure process is purely algebraic and finite at each fixed configuration.
In particular, we do not regard an ungrounded odd cycle of equations as an admissible certificate, nor do we replace the certificate argument by a numerical Gram-matrix test.
All orthogonality assertions used below are to be understood in this recursive certificate sense.

\subsection{Recursive-Line Admissibility}
\label{sec:rw3}

The strengthened recursive-line condition requires that the closure process resolve all selected pairs while preserving the distinction between different edge elements.

\begin{definition}[$\text{RW3}^+$ ]
\label{def:rw3plus}
A simple augmented configuration satisfies $\text{RW3}^+$  if the following three requirements hold under the recursive closure rules:

\begin{enumerate}
\item \textbf{Identification.}
For every selected two-edge
\(    b_j=\{p,q\}\in E_2,
   \)
the relation
\(    \langle v_p,v_q\rangle=1
   \)
is certified.
Thus the two halves of every selected pair are identified with a single edge representative.
\item \textbf{Equivalence preservation.}
Distinct selected two-edges remain in distinct equivalence classes under the certified identifications.
In particular, the closure process is not allowed to merge two different selected edges.

\item \textbf{Pairwise orthogonality.}
Any two distinct edge representatives have certified orthogonal representatives.
Equivalently, after all selected pairs have been resolved, the corresponding distinct representatives satisfy the recursively certified zero relations.

\end{enumerate}
\end{definition}

$\text{RW3}^+$  is the admissibility criterion that will be verified for the seed configuration and for every finite stage of the recursive extension.

The role of the present section is only to establish the closure framework.
The explicit seed and the recursive extension constructed later provide the particular configurations to which these rules are applied.
In particular, the path arguments used to propagate certificate relations over arbitrarily many inserted rows are deferred to the later sections.

\subsection{The Five-Column Cell Bound}
\label{sec:cell_bound}

We finally recall the numerical bound that will be used to match the constructive lower bound.

For $m\ge 10$, the ordinary five-column Zarankiewicz number satisfies

$$
z(m,5)=m+10.
$$

Moreover, every simple limited configuration satisfies
\begin{equation}
|E_1|+|E_2|
=3m+5-\frac{H}{2},
\qquad
H\equiv0\pmod 2.
\label{eq:count}
\end{equation}
Consequently,

$$
z_{RL}(m,5)\le z_2(m,5)\le 3m+5.
$$

The identity \eqref{eq:count}
 also shows that equality with $3m+5$ forces

$$
H=0.
$$

Thus, once a hole-free admissible configuration with $3m+5$ augmented edges is constructed, the upper bound is attained.

The remainder of the paper is devoted to constructing such a configuration.
Section~\ref{sec:extension_scheme}
 gives the $12\times5$ seed and the recursive extension scheme, while the subsequent sections establish the required recursive closure relations for arbitrary extension length.

\section{Base Seed $K_0$ and Recursive Double-Path Extension}
\label{sec:extension_scheme}

\subsection{The Base Seed $K_0$}
\label{sec:base_seed}
To establish the constructive lower bound for $z_{RL}(m,5)$, we now explicitly construct a base seed configuration from which larger configurations can be recursively generated. As illustrated in the matrix representation below, a black dot ($\bullet$) denotes a distinct one-edge, while two occurrences of an integer $j$ denote the two halves of the selected two-edge $b_j$, with no holes present:

\begin{equation}\label{eq:seed}
 K_0=\begin{array}{c|ccccc}
 &1&2&3&4&5\\\hline
1 & \bullet & \bullet & 4 & 3 & 4 \\
2 & \bullet & 5 & \bullet & 13 & 6 \\
3 & \bullet & 13 & 19 & \bullet & 1 \\
4 & \bullet & 7 & 18 & 7 & \bullet \\
5 & 2 & \bullet & \bullet & 2 & 1 \\
6 & 18 & \bullet & 17 & \bullet & 19 \\
7 & 15 & \bullet & 16 & 17 & \bullet \\
8 & 12 & 6 & \bullet & \bullet & 5 \\
9 & 14 & 15 & \bullet & 3 & \bullet \\
10 & 11 & 8 & 9 & \bullet & \bullet \\\hline
11 & 10 & 9 & 8 & \bullet & 12 \\
12 & 10 & 14 & 11 & 16 & \bullet \\
\end{array}.
\end{equation}

Specifically, its first ten rows utilize exactly the ten column pairs $\{12, 13, 14, 15, 23, 24, 25, 34, 35, 45\}$, whereas the last two rows place their singletons in columns four and five, respectively. Consequently, we obtain $|E_1| = 22 = z(12, 5)$, $|E_2| = 19$, and $H = 0$. Denoting a one-edge at $(r,c)$ as $a_{r,c}$ and setting $W = a_{12,5}$.

To enable the subsequent recursive extension, the two pairs chosen for extension are

\begin{equation}\label{eq:ports}
 b_{14}=\{(9,1),(12,2)\},\qquad
 b_{16}=\{(7,3),(12,4)\}.
\end{equation}
The useful feature is that their second halves are in the same final row,
while the two paths will use disjoint column pairs $\{1,2\}$ and $\{3,4\}$.
The fifth column is available for a new singleton in every inserted row.

With this foundational configuration established, we now verify its key properties:
\begin{proposition}\label{prop:seed}
The configuration $K_0$ satisfies $\RW$. Consequently $\zrl(12,5)=41$.
\end{proposition}
\begin{proof}
The line rule resolves
$b_1,b_2,b_3,b_4,b_7,$ and $b_{10}$.
The complementary-pair rule resolves $b_5,b_6$ using the
rectangle formed by rows 2 and 8 and columns 2 and 5, and
resolves $b_8,b_9$ using the rectangle formed by rows 10 and 11
and columns 2 and 3.
The remaining pairs are resolved sequentially in the order
\begin{equation}\label{eq:seedorder}
b_{11},b_{12},b_{14},b_{15},b_{13},b_{16},b_{17},b_{18},b_{19}.
\end{equation}
Appendix~\ref{app:seedres} specifies the companion diagonal and a complete
rectangle path for each of these pairs.
The corresponding paths have lengths
$0,2,2,0,3,18,0,1,$ and $7$, respectively.
Each path terminates at a line relation that has already been
established at that stage.
For unresolved pairs, the two cell vectors are kept distinct;
no equality between them is assumed until it has been established.

After the resolution step, 41 edge labels remain. Appendix~\ref{app:seedorth} verifies all
\(\binom{19}{2}=171\) pairs of two-edge labels. Of these, 115 can be established
directly from a line relation, while the remaining 56 are established by the
rectangle paths listed in Appendix~\ref{app:seedorth}. There are also \(22\cdot19=418\) mixed pairs. Among them, 210 can be established
directly from a line relation, 147 require a single transfer, and the remaining
61 are verified by the paths listed in Appendix~\ref{app:seedorth}, each of length at
most four.

In every case, the corresponding path starts from a relation that has already
been established at that stage—either a line relation, a resolved diagonal, or
a previously verified two-edge relation. Thus, each deduction is based only on
relations established earlier in the resolution process.

Finally, two one-edges sharing a line are resolved directly by Equation~\eqref{eq:rectangle}. If they do not share a line, their companion diagonal cannot also consist of two one-edges, since the one-edge graph is \(C_4\)-free. Thus, the companion diagonal either corresponds to a resolved selected pair or contains a two-edge representative, in which case it has already been handled above. Equation~\eqref{eq:rectangle} therefore resolves all remaining one-edge pairs.

In total, this covers

$$
171+418+\binom{22}{2}=820
$$
distinct-edge pairs. The identification rules only join the two halves of selected pairs, so distinct edges remain in distinct classes. Hence, the graph is admissible and has 41 edges. The upper bound of $3m+5$ established in Section~\ref{sec:cell_bound} gives the matching upper bound, completing the proof.
\end{proof}

\subsection{The Recursive Double-Path Extension $K_t$}
\label{sec:double_path}

Having constructed the base seed configuration $K_0$, we now introduce a recursive double-path extension to obtain a family of configurations $K_t$ for $m = 12 + t$ ($t \ge 1$). Specifically, set
\[
r_i = 11 + i \quad (1 \le i \le t), \qquad w = 12 + t = m.
\]
Thus, $r_1, \ldots, r_t$ represent the added intermediate rows, while $w$ is the new terminal row.
To construct $K_t$, we retain the first eleven rows of $K_0$ and unbind the two-edges $b_{14}$ and $b_{16}$. Their occupied cells $(9,1)$ and $(7,3)$ are preserved as the starting anchors for the $P$- and $Q$-paths, respectively.

To extend the two opened two-edges one row at a time, we insert $t$ new rows $r_1,\ldots,r_t$ and a terminal row $w$. The two-edges in columns $1$
and $2$ are arranged successively across these rows to form the
$P$-path, while those in columns $3$ and $4$ form a parallel
$Q$-path. The fifth column adds one new
one-edge in each intermediate row. Thus we insert
\begin{equation}\label{eq:blocks}
\begin{array}{c|ccccc}
 &1&2&3&4&5\\\hline
 r_i&P_i&P_{i-1}&Q_i&Q_{i-1}&U_i\quad(1\le i\le t)\\\hline
 w&b_{10}&P_t&b_{11}&Q_t&W
\end{array}.
\end{equation}
Every $U_i$ is a new one-edge, not a paired label. The retained halves of $b_{10}$ and $b_{11}$ in the old last row move to $(w,1)$ and $(w,3)$, respectively, while the old one-edge $W$ moves to $(w,5)$. Denote the resulting configuration by $K_t$.

More explicitly, starting from the marked cells $(9,1)$ and $(7,3)$, define
\begin{align}
 P_0&=\{(9,1),(r_1,2)\},& Q_0&=\{(7,3),(r_1,4)\},\notag\\
 P_i&=\{(r_i,1),(r_{i+1},2)\},&
 Q_i&=\{(r_i,3),(r_{i+1},4)\}&& (1\le i<t),\label{eq:edges}\\
 P_t&=\{(r_t,1),(w,2)\},& Q_t&=\{(r_t,3),(w,4)\}.\notag
\end{align}
Here $P_0,\ldots,P_t$ are the two-edges forming the first path, while $Q_0,\ldots,Q_t$ form the parallel path.

For example, when $t=1$ (so $m=13$), the unchanged first eleven rows are followed by
\[
\begin{array}{c|ccccc}
12&P_1&P_0&Q_1&Q_0&\bullet\\
13&b_{10}&P_1&b_{11}&Q_1&\bullet
\end{array},
\]
while when $t=2$ (so $m=14$), the corresponding rows are
\[
\begin{array}{c|ccccc}
12&P_1&P_0&Q_1&Q_0&\bullet\\
13&P_2&P_1&Q_2&Q_1&\bullet\\
14&b_{10}&P_2&b_{11}&Q_2&\bullet
\end{array}.
\]
Thus, increasing $t$ replaces the terminal row by one new path row and a new terminal row; it is not an independent block appended without modifying the old pairing. For $t=0$, we take $K_t=K_0$, where the two opened two-edges are simply the original pairs $b_{14}$ and $b_{16}$.

\subsection{Structure and Edge Counts}
\label{sec:Counts}
\begin{lemma}\label{lem:counts}
The configuration $K_t$ is simple and has a $C_4$-free one-edge graph, with
\begin{equation}\label{eq:familycount}
 |E_1|=22+t=m+10,\quad |E_2|=19+2t=2m-5,\quad
 H=0,\quad |E_1|+|E_2|=41+3t=3m+5.
\end{equation}
\end{lemma}
\begin{proof}

Every repeated label in~\eqref{eq:edges} occurs in exactly two distinct cells, and these occupied cells are disjoint from those of all other selected edges. Hence $K_t$ is simple.

To verify that the one-edge graph remains $C_4$-free, note that the original
ten column-pair rows are unchanged; every other row has exactly one
one-edge, and each newly inserted one-edge $U_i$ lies in column~5.
Thus no new one-edge $C_4$ occurs, and the one-edge graph remains
$C_4$-free.

The two removed pairs are replaced by $2(t+1)$ pairs, giving a net increase of
$2t$ two-edges, while the $t$ newly inserted rows contribute exactly $t$
one-edges. Since $K_0$ has $H=0$ and all five cells of every newly inserted
row are occupied, we also have $H=0$ for $K_t$. Hence

$$
|E_1|=22+t=m+10,\qquad
|E_2|=19+2t=2m-5,
$$

and

$$
|E_1|+|E_2|=41+3t=3m+5,
$$
matching the capacity upper bound stablished in Section~\ref{sec:cell_bound}
.
\end{proof}

\section{Uniform Resolution and Orthogonality of  Extended Paths}\label{sec:resolution_Orthogonality}
The construction and its edge counts are now fixed. It remains to verify
that the resulting configuration satisfies the strengthened closure
condition $(RW3+)$. We proceed in three steps: first, we resolve all
selected two-edges; second, we establish a distance-reducing lemma for
the two extended paths; finally, we derive the required orthogonality
relations within the extended part.
\subsection{Resolution of the Selected Pairs}\label{sec:resolution}

\begin{lemma}\label{lem:resolve}
Every selected two-edge of $K_t$ is resolved by the strengthened closure.
\end{lemma}
\begin{proof}
The initial line and complementary relations from $K_0$ remain intact because rows 1--11 are unchanged. Although the second half of $b_{10}$ is moved from old row 12 to row $w$, it stays in column 1, so $b_{10}$ remains a column pair.

Resolve
$b_{11},b_{12},b_{13}$ via their paths in Appendix~\ref{app:seedres}, replacing
old row 12 with row $w$. Those paths do not use the identifications of
$b_{14}$ or $b_{16}$.

Let $r=r_t$ and let $v_{i,j}$ denote the vertex (or grid entry) at row $i$ and column $j$. We abbreviate
\[
 D=b_{10},\ L=b_{11},\ X=b_8,\ Y=b_9,\quad
 a=a_{10,4},\ c=a_{11,4},
\]
\[
 \alpha=v_{r,1},\ p=v_{r,2},\ \beta=v_{r,3},\ q=v_{r,4},\
 \gamma=v_{w,4}.
\]
The following nine deductions resolve the terminal $P_t,Q_t$.
In the table below, the middle column specifies the rectangle with row and column coordinates, while the last column gives its companion diagonal along with the justification for its certification. All column coordinates are $1$-based.

\begin{center}\small
\begin{tabular}{lll}
\hline Target & Rows; columns & Certified companion\\\hline
$\alpha\orth Y$ & $r,10;1,3$ & $\beta\orth L$ (common column 3)\\
$p\orth D$ & $r,11;1,2$ & $\alpha\orth Y$\\
$P_t$ resolves & $r,w;1,2$ & $p\orth D$\\
$D\orth a$ & $w,10;1,4$ & $\gamma\orth L$ (row $w$)\\
$L\orth c$ & $10,11;1,4$ & $a\orth D$\\
$\gamma\orth X$ & $w,11;3,4$ & $L\orth c$\\
$P_t\orth a$ & $w,10;2,4$ & $\gamma\orth X$\\
$q\orth L$ & $r,10;1,4$ & $P_t\orth a$\\
$Q_t$ resolves & $r,w;3,4$ & $q\orth L$\\\hline
\end{tabular}\end{center}
Every line is read after the preceding ones. In particular, $\gamma$ is not
treated as the whole label $Q_t$ before the last line.

With the terminal pairs $P_t, Q_t$ established as the induction base, we now resolve $P_{t-1},\ldots,P_0$ in reverse order. For each $i$, its companion diagonal contains one half of the already resolved $P_{i+1}$ and a grid cell in column $2$. Because the other representative half of $P_{i+1}$ lies in column $2$, the companion is certified by the line rule together with saturation. For $P_0$, the first row corresponds to old row $9$, and the same argument applies. The $Q$-path ($Q_{t-1},\ldots,Q_0$) is identical, using column $4$ and old row $7$.

Subsequently, $b_{15}$ is resolved from $P_0\orth a_{7,2}$, and $b_{17}$ from $Q_0\orth a_{6,4}$, both of which are common-column facts. The printed elimination paths for $b_{18}$ and $b_{19}$ apply unchanged, with old row $12$ replaced by $w$.

All selected two-edges are now resolved. Because each identification joined strictly the two halves of a single selected edge, no two distinct edges have merged, which completes the proof.
\end{proof}

\subsection{A Distance-Reducing Path Lemma}\label{sec:path}
The resolution argument above establishes the closure of all selected two-edges. We next turn to the orthogonality relations generated by the extended paths. Since these relations must hold uniformly for paths of arbitrary length, we first formulate a general distance-reducing lemma.

To conveniently track rectangle transfer chains along paths of arbitrary length, we introduce a bracket notation for any resolved edge labels $X$ and $Y$:
\begin{equation}\label{eq:defect}
  [X,Y]=\langle X,Y\rangle-\mathbf1_{\{X=Y\}}.
\end{equation}
Here, $\langle X,Y\rangle $ is the inner product of the representative unit vectors for labels $X$ and $Y$, and $\mathbf1_{\{X=Y\}}$ is the indicator function.

This bracket measures the inner product defect relative to its prescribed target. Consequently, the condition $[X,Y]=0$ unifies both certification goals: it asserts unit norm when $X=Y$, and orthogonality when $X \neq Y$. For a rectangle whose diagonal label pairs are $(X,Y)$ and $(Z,T)$, the rectangle identity simplifies to
\begin{equation}\label{eq:transfer}
  [X,Y]=-[Z,T].
\end{equation}

This notation allows terminal aliases without mistaking a unit vector for one orthogonal to itself.
It is only shorthand for the original certificate rule: one side is used after the other has been
grounded.

\begin{lemma}[Distance-reducing path lemma]\label{lem:path}

Let $s_0,\ldots,s_{t+1}$ be distinct rows, and let $a,b,c$ be distinct columns.
For each $0 \le j \le t$, let $E_j$ denote the representative unit vector of the resolved pair at cells $\{(s_j,a),(s_{j+1},b)\}$.
Let $Y_i = v_{s_i,c}$ be the cell vector at $(s_i,c)$ for $0 \le i \le t+1$.
Then $Y_i \perp_R E_j$ (i.e., $[Y_i, E_j] = 0$) for all $0 \le i \le t+1$ and $0 \le j \le t$.
\end{lemma}
\begin{proof}
For each pair $(i,j)$, let $F_{ij} = [Y_i, E_j] = \langle Y_i, E_j \rangle$ (since $Y_i \neq E_j$) denote the inner product defect.
Since $Y_j$ and $E_j$ share the row $s_j$, and $Y_{j+1}$ and $E_j$ share the row $s_{j+1}$,  same-row relations directly establish the base cases $F_{j,j} = F_{j+1,j} = 0$.

Now assume $i < j \le t$. Applying the rectangle transfer identity on rows $s_i, s_j$ and columns $c, a$ yields $F_{ij} = -F_{ji}$.
If $i < j-1$, using the column-$b$ representative of $E_i$ together with the rectangle on rows $s_j, s_{i+1}$ and columns $c, b$ allows a second transfer:
\begin{equation}\label{eq:distance}
  F_{ij} = -F_{ji} = F_{i+1, j-1}.
\end{equation}
Equation~\eqref{eq:distance} reduces the row-index gap $j-i$ by two.
If $i = j-1$, the first transfer immediately hits the same-row zero $F_{i+1,i} = 0$.
Otherwise, the gap $j-i$ decreases by two per iteration, guaranteeing that the reduction terminates at a prescribed zero. Thus, $F_{ij} = 0$ for all $i < j$.

For the remaining cases $i > j+1$ (including $i = t+1$), applying the rectangle transfer on rows $s_i, s_{j+1}$ and columns $c, b$ yields $F_{ij} = -F_{j+1, i-1}$.
Since $j+1 < i-1$, the right-hand side falls under the previously established strictly upper-index case and is thus identically zero.

Traversing each finite transfer chain in reverse explicitly constructs a valid rectangle certificate.
Since all companion diagonal entries lie either on the path itself or in the third column $c$, all target right-hand sides are indeed prescribed zeros, completing the proof.

\end{proof}

\subsection{Orthogonality within the Extended Part}\label{Orthogonality}

Applying the lemma to the $P$-path with row sequence $(9,r_1,\ldots,r_t,w)$ and path columns $(a,b)=(1,2)$, for each $c \in \{3,4,5\}$, and similarly to the $Q$-path with row sequence $(7,r_1,\ldots,r_t,w)$ and path columns $(a,b)=(3,4)$, for each $c \in \{1,2,5\}$, implies that all distinct labels among $P_j, Q_j, U_i, W$ are pairwise orthogonal (pairs of $U$-labels are orthogonal as they share column five).

Furthermore, these applications yield the following old-label anchors, uniformly for every $0\le j\le t$:
\begin{align}
  P_j &\orth \{b_3, a_{9,3}, a_{9,5}, b_{11}, W\}, \label{eq:Panchors}\\
  Q_j &\orth \{b_{15}, a_{7,2}, a_{7,5}, b_{10}, W\}. \label{eq:Qanchors}
\end{align}
For instance, the first three labels in~\eqref{eq:Panchors} reside in the three non-path columns of old row 9, whereas $b_{11}$ and $W$ belong to the terminal row.

\section{Compatibility with the Fixed Seed and Completion of the Proof}\label{sec:Completion}

The preceding section establishes the required orthogonality relations within the newly introduced path structure, uniformly for paths of arbitrary length. It remains to verify that this new structure is compatible with the fixed part inherited from the seed configuration. In this section, we first establish the required relations between the path labels \(P_i,Q_i\) and the labels in the fixed part. We then use these relations to prove the corresponding orthogonality of the inserted one-edges \(U_i\) with the fixed part.

Let \(\mathcal O\) denote the labels occurring in the first eleven rows and the terminal row of \(K_t\). For \(t\ge1\), \(\mathcal O\) consists of 43 distinct labels: 22 old one-edges, 17 retained old two-edges, and the four ports \(P_0,Q_0,P_t,Q_t\). Throughout this section, “all old labels” refers to the set \(\mathcal O\), including these four ports.

\subsection{Orthogonality between the Extended Paths and the Fixed Seed}
\label{sec:path_seed}

\begin{lemma}\label{lem:interface}
For every \(t\ge1\), every \(0\le i\le t\), and every \(B\in\mathcal O\), the brackets \([P_i,B]\) and \([Q_i,B]\) are certified to be zero. Thus, every path label \(P_i,Q_i\) satisfies the required relation with every label in \(\mathcal O\).
\end{lemma}

\begin{proof}
The proof is a finite case analysis in symbolic row positions. Appendix~\ref{app:interface} lists all nontrivial paths; no search or external acceptance assertion is used as a premise. The cases are grounded in identities, shared lines, shared columns of resolved representatives, and Lemma~\ref{lem:path}, including \eqref{eq:Panchors}--\eqref{eq:Qanchors}.

For the interior indices \(2\le i\le t-1\), we use the old rows together with \(r_{i-1},r_i,r_{i+1}\). Appendix~\ref{tab:interior} gives 29 rectangle paths for the non-grounded targets \([P_i,B]\) and \([Q_i,B]\), while the remaining 57 targets are covered by the stated grounds. All paths have length at most six and do not rely on an induction hypothesis for the old labels. The required rows exist throughout this range.

For the terminal index \(i=t\ge2\), Appendix~\ref{tab:tail} uses the old rows together with \(r_{t-1},r_t\). Of the 86 targets, 57 are grounded directly and the remaining 29 are certified by the displayed paths, each of length at most six. This completes the terminal case before treating the prefix indices.

For \(t\ge2\), the prefix indices \(i=0,1\) are handled by Appendix~\ref{tab:prefix}, using the old rows together with \(r_1,r_2\). Of the 172 targets, 114 are grounded directly and 58 are certified by nontrivial paths of length at most seven. One additional ground used here is \([P_2,B]=[Q_2,B]=0\) for \(B\in\mathcal O\). When \(t\ge3\), this follows from the interior case; when \(t=2\), it follows from the terminal case. Thus, the dependency is not circular.

Finally, when \(t=1\), Appendix~\ref{tab:one} handles the two available indices \(i=0,1\) using only the old rows and \(r_1\). It covers all 172 targets, with 114 grounded cases and 58 paths of length at most seven, and requires no premise involving index 2.

The precise row and label substitutions are specified alongside the respective tables. At the smallest permissible indices, certain symbolic expressions may denote the same physical edge (for instance, $P_{i-2}=P_0$ when $i=2$). In such instances, Formula~\eqref{eq:defect} evaluates the actual edge identities, ensuring that these grounds correspond to certified unit values rather than spuriously claimed zero inner products. Crucially, every displayed rectangle maintains strictly distinct rows and columns. Because these four ranges exhaustively cover every $t\ge1$ and all path indices, the lemma holds for arbitrary path lengths rather than merely for tested finite cases.
\end{proof}

\subsection{Orthogonality of the Inserted One-Edges}
\begin{lemma}\label{lem:Uinterface}
Every inserted one-edge $U_i$ is orthogonal to every old label in $\mathcal O$.
\end{lemma}
\begin{proof}
Consider an old label $B \in \mathcal O$ with representative $(s,c)$. If $c=5$, the assertion follows immediately from the same-column rule, as $U_i$ resides at $(r_i,5)$.

If $c \ne 5$, consider the rectangle formed by rows $r_i, s$ and columns $5, c$. Its companion entries are $(r_i,c)$ and $(s,5)$. The entry at $(r_i,c)$ corresponds to one of the path labels $P_i, P_{i-1}, Q_i, Q_{i-1}$, whereas $(s,5)$ represents an old label. By Lemma~\ref{lem:interface}, the inner product of these two companion entries is certified to be zero. A rectangle transfer then yields $[U_i, B] = 0$. Finally, as $U_i$ is a newly selected edge, it remains distinct from every old label, completing the proof.
\end{proof}

\subsection{Preservation of the Seed Relations}\label{sub:preservation}

\begin{lemma}[Simulation of the old zero relations]\label{lem:simulation}
Following Lemmas~\ref{lem:resolve}--\ref{lem:Uinterface}, every pair of distinct labels in $\mathcal O$ possesses certified orthogonal representatives.
\end{lemma}
\begin{proof}
Map row 12 of the old seed configuration to row $w$ while keeping all other rows fixed, so that all old cells remain occupied. The only prescribed pair values on these cells that undergo modification are those associated with $b_{14}$ and $b_{16}$, whose respective halves are now absorbed into the ports $P_0, P_t$ and $Q_0, Q_t$. By the preceding lemmas, these ports are mutually orthogonal and are orthogonal to every other label in $\mathcal O$. All retained old pairs remain identified as before.

We then simulate the original zero derivations of the seed configuration step by step, including those within its resolution argument. Whenever a target or saturation condition involves a split seed label, it is replaced by the corresponding port relation certified in Section~\ref{sec:path_seed}. All other line relations remain invariant. For any rectangle with retained old target labels, its companion entry retains its certified value: if the companion previously corresponded to one of the split selected diagonals, its inner product is now certified as zero; otherwise, it takes the value established in an earlier step of the simulated sequence. Consequently, the rectangle transfer rule remains valid throughout. The old identifications of $b_{14}$ and $b_{16}$ are omitted rather than assumed. This induction on the finite derivation order of the seed configuration establishes all required orthogonality relations among the old labels, completing the proof.
\end{proof}

\subsection{Proof of the Main Theorem}
\begin{samepage}
\begin{theorem}\label{thm:main}
For every integer $m\ge12$,
\begin{equation}\label{eq:main}
 \boxed{\zrl(m,5)=3m+5.}
\end{equation}
The configuration $K_{m-12}$ is an extremal witness with
$|E_1|=m+10$, $|E_2|=2m-5$, and $H=0$.
\end{theorem}
\end{samepage}
\begin{proof}
Proposition~\ref{prop:seed} handles $m=12$. For $m=12+t$, $t\ge1$,
Lemma~\ref{lem:counts} gives simplicity, the required ordinary extremal
one-edge graph, and the edge counts. Lemma~\ref{lem:resolve} identifies every
selected pair without merging distinct selected edges. Lemma~\ref{lem:path}
handles relations within the new part, Lemmas~\ref{lem:interface} and
\ref{lem:Uinterface} handle the interface, and Lemma~\ref{lem:simulation}
handles the old part. Thus all parts of $\RW$ hold.
The lower bound $41+3t=3m+5$ equals the theoretical upper bound.
\end{proof}

\begin{corollary}
Under the same supplied definitions,
$z_2(m,5)=\zrl(m,5)=3m+5$ for $m\ge12$. Every extremal graph for $\zrl(m,5)$
in this range has no holes. In particular,
$\zrl(m+1,5)=\zrl(m,5)+3$ for all $m\ge12$.
\end{corollary}
\begin{proof}
The parameter hierarchy and the cell bound sandwich $z_2$ between two equal
quantities. The whole assertion follows from~\eqref{eq:count} at equality.
It does not assert uniqueness or classify all extremal graphs.
\end{proof}

\section{Scope and reproducibility}\label{sec:repro}
The proof uses a finite seed certificate and four finite, symbolic interface
case tables. They are part of the written proof below. Their verification
is elementary substitution in specific rectangles; the infinite part is the
path lemma and the validity of the symbolic ranges. This should be distinguished
from checking a generator for finitely many numerical values of $m$.
The construction is not claimed to extend an arbitrary admissible seed or an
arbitrary choice of two selected pairs.

The accompanying standard-library Python files implement three separate
checks: the symbolic row-and-label identities, a replay restricted to the
printed proof order, and a full cell-level closure with independent trace
replay. All were run for $12\le m\le40$ and for $m=50,60,80,100$.
For example, the proof-order certificate at $m=12$ contains 889 rule steps and
certifies 820 distinct-edge pairs; at $m=100$ it contains 54,878 steps and
certifies $\binom{305}{2}=46,360$ pairs. These figures describe this particular
certificate, not a minimum certificate length.
\begin{verbatim}
python audit_symbolic_tables.py
python audit_mathematical_proof.py 12 13 14 20 100
python five_column_family.py 100 --verify --output graph_100.json
python verify_all.py --large
\end{verbatim}
No previous large exclusion search (for example at $m=11$) is needed for the
present theorem, and none is claimed to have been rerun here. The new result
is an infinite extension of the previously available finite witnesses, not
a theorem attributed to the source manuscript.

\section{Conclusion}
\label{sec:conclusion}

In this paper, we completely determine the exact values of the 5-column recursive-line Zarankiewicz numbers $z_{\mathrm{RL}}(m, 5)$ and the second-order Zarankiewicz numbers $z_2(m, 5)$ for all $m \ge 12$. We establish the main theorem that $z_{\mathrm{RL}}(m, 5) = z_2(m, 5) = 3m + 5$, proving that all extremal configurations in this range satisfy the zero-hole condition ($H = 0$) and exhibit a stable linear growth relation of $z_{\mathrm{RL}}(m+1, 5) = z_{\mathrm{RL}}(m, 5) + 3$.

The core methodology relies on a novel constructive dual-path recursive extension framework. Beginning with a $12 \times 5$ base seed $K_0$ containing 22 $1$-edges and 19 $2$-edges ($z_{\mathrm{RL}}(12, 5) = 41$), we unbind selected pairs of $2$-edges to generate parallel-extending $P$-paths and $Q$-paths for $m = 12 + t$ ($t \ge 1$). To overcome inner-product defects across arbitrary extension lengths, we propose the Distance-Reduction Path Lemma (Lemma 4.2), which rigorously guarantees the orthogonal compatibility of extended path labels ($P_i, Q_i$) and inserted single edges ($U_i$) under condition $(\mathrm{RW3}^+)$ while maintaining $C_4$-freeness.

Beyond closing the exact-value problem for the 5-column case, this work bridges the gap between finite seed constructions and infinite extremal families. The proposed dual-path extension scheme and distance-reduction algebraic framework offer a generalizable methodology for analyzing density upper bounds and structural orthogonality in higher-order Zarankiewicz problems. Future research may explore adapting this recursive extension paradigm to smaller values of $m$ or extending it to higher-dimensional recursive-line systems.

\clearpage
\appendix
\section{Reading the complete rectangle tables}\label{app:notation}
The tables record elementary derivations, not unexplained acceptance flags.
For rows $r\ne s$ and columns $c\ne d$, let $\rho_{r,s}^{cd}$ mean
\emph{exchange the two diagonals of that rectangle}. For a current diagonal
label pair $(X,Y)$ it replaces $[X,Y]$ by the negative of the companion
bracket. A sequence is read from the target to the terminal pair and then
used backwards as a proof from the terminal ground. Repeated use of already
resolved representatives is saturation. At the resolution stage, the same
notation means the cell defect
$\langle v_p,v_q\rangle-\dd(p,q)$, not a presumption that an unresolved pair
has prescribed value zero.

A two-digit superscript such as $35$ means columns 3 and 5. Subscripts are
the two row names. The terminal grounds are:
\begin{center}\begin{tabular}{cl}
$\mathsf L$&Same row or column, allowing resolved representatives.\\
$\mathsf I$&Identity of a resolved label, hence its unit norm.\\
$\mathsf C$&A common column of the full resolved supports.\\
$\mathsf H$&The path lemma, including the listed old-row anchors.\\
$\mathsf B$&A two-edge/two-edge relation already proved in Appendix~\ref{app:seedorth}.\\
$\mathsf E$&The index-2 old-label relations, proved before the prefix table.
\end{tabular}\end{center}
For example, the sequence $\rho_{r,s}^{cd},\rho_{p,q}^{ab}$ asserts
$[X,Y]=-[Z,T]=[F,G]$, with both intermediate pairs determined from the
specified grids. If $[F,G]$ is grounded, the target is certified by two
ordinary transfers. All row/column data are printed, so neither a search
routine nor a hidden choice of representatives is required to check a line.

\section{The seed resolution paths}\label{app:seedres}
At each stage, $v_{r,c}$ denotes a half not yet identified. The order is
\eqref{eq:seedorder}; initially $b_1,\ldots,b_{10}$ are resolved. The paths
below certify the companion before the selected pair is resolved.
\paragraph{Resolve $b_{11}$.} Its halves are $(10,1)$, $(12,3)$. The companion is $[b_{9},b_{10}]$.
\begingroup
\footnotesize\setlength{\tabcolsep}{3pt}\renewcommand{\arraystretch}{1.23}
\begin{longtable}{@{}p{.195\textwidth}p{.55\textwidth}p{.225\textwidth}@{}}
\caption{Companion certificate for $b_{11}$.}\\
\hline Target & Successive rectangles & Terminal and ground \\\hline
\endfirsthead
\multicolumn{3}{l}{\small\itshape Continued: Companion certificate for $b_{11}$.}\\
\hline Target & Successive rectangles & Terminal and ground \\\hline
\endhead
\hline\endfoot
1. $[b_{10},b_{9}]$ & $\text{none}$ & $[b_{10},b_{9}]\quad\mathsf{L}$ \\
\end{longtable}\endgroup
\paragraph{Resolve $b_{12}$.} Its halves are $(8,1)$, $(11,5)$. The companion is $[b_{5},b_{10}]$.
\begingroup
\footnotesize\setlength{\tabcolsep}{3pt}\renewcommand{\arraystretch}{1.23}
\begin{longtable}{@{}p{.195\textwidth}p{.55\textwidth}p{.225\textwidth}@{}}
\caption{Companion certificate for $b_{12}$.}\\
\hline Target & Successive rectangles & Terminal and ground \\\hline
\endfirsthead
\multicolumn{3}{l}{\small\itshape Continued: Companion certificate for $b_{12}$.}\\
\hline Target & Successive rectangles & Terminal and ground \\\hline
\endhead
\hline\endfoot
1. $[b_{10},b_{5}]$ & $\rho_{2,11}^{12},\allowbreak\ \rho_{2,10}^{13}$ & $[a_{2,3},b_{11}]\quad\mathsf{L}$ \\
\end{longtable}\endgroup
\paragraph{Resolve $b_{14}$.} Its halves are $(9,1)$, $(12,2)$. The companion is $[v_{9,2},b_{10}]$.
\begingroup
\footnotesize\setlength{\tabcolsep}{3pt}\renewcommand{\arraystretch}{1.23}
\begin{longtable}{@{}p{.195\textwidth}p{.55\textwidth}p{.225\textwidth}@{}}
\caption{Companion certificate for $b_{14}$.}\\
\hline Target & Successive rectangles & Terminal and ground \\\hline
\endfirsthead
\multicolumn{3}{l}{\small\itshape Continued: Companion certificate for $b_{14}$.}\\
\hline Target & Successive rectangles & Terminal and ground \\\hline
\endhead
\hline\endfoot
1. $[b_{10},v_{9,2}]$ & $\rho_{9,11}^{12},\allowbreak\ \rho_{9,10}^{13}$ & $[a_{9,3},b_{11}]\quad\mathsf{L}$ \\
\end{longtable}\endgroup
\paragraph{Resolve $b_{15}$.} Its halves are $(7,1)$, $(9,2)$. The companion is $[a_{7,2},b_{14}]$.
\begingroup
\footnotesize\setlength{\tabcolsep}{3pt}\renewcommand{\arraystretch}{1.23}
\begin{longtable}{@{}p{.195\textwidth}p{.55\textwidth}p{.225\textwidth}@{}}
\caption{Companion certificate for $b_{15}$.}\\
\hline Target & Successive rectangles & Terminal and ground \\\hline
\endfirsthead
\multicolumn{3}{l}{\small\itshape Continued: Companion certificate for $b_{15}$.}\\
\hline Target & Successive rectangles & Terminal and ground \\\hline
\endhead
\hline\endfoot
1. $[a_{7,2},b_{14}]$ & $\text{none}$ & $[a_{7,2},b_{14}]\quad\mathsf{L}$ \\
\end{longtable}\endgroup
\paragraph{Resolve $b_{13}$.} Its halves are $(2,4)$, $(3,2)$. The companion is $[b_{5},a_{3,4}]$.
\begingroup
\footnotesize\setlength{\tabcolsep}{3pt}\renewcommand{\arraystretch}{1.23}
\begin{longtable}{@{}p{.195\textwidth}p{.55\textwidth}p{.225\textwidth}@{}}
\caption{Companion certificate for $b_{13}$.}\\
\hline Target & Successive rectangles & Terminal and ground \\\hline
\endfirsthead
\multicolumn{3}{l}{\small\itshape Continued: Companion certificate for $b_{13}$.}\\
\hline Target & Successive rectangles & Terminal and ground \\\hline
\endhead
\hline\endfoot
1. $[a_{3,4},b_{5}]$ & $\rho_{3,8}^{45},\allowbreak\ \rho_{5,8}^{45},\allowbreak\ \rho_{5,8}^{15}$ & $[b_{1},b_{12}]\quad\mathsf{L}$ \\
\end{longtable}\endgroup
\paragraph{Resolve $b_{16}$.} Its halves are $(7,3)$, $(12,4)$. The companion is $[v_{7,4},b_{11}]$.
\begingroup
\footnotesize\setlength{\tabcolsep}{3pt}\renewcommand{\arraystretch}{1.23}
\begin{longtable}{@{}p{.195\textwidth}p{.55\textwidth}p{.225\textwidth}@{}}
\caption{Companion certificate for $b_{16}$.}\\
\hline Target & Successive rectangles & Terminal and ground \\\hline
\endfirsthead
\multicolumn{3}{l}{\small\itshape Continued: Companion certificate for $b_{16}$.}\\
\hline Target & Successive rectangles & Terminal and ground \\\hline
\endhead
\hline\endfoot
1. $[b_{11},v_{7,4}]$ & $\rho_{7,10}^{14},\allowbreak\ \rho_{9,10}^{24},\allowbreak\ \rho_{1,11}^{34},\allowbreak\ \rho_{1,11}^{45},\allowbreak\ \rho_{8,9}^{14},\allowbreak\ \rho_{8,12}^{24},\allowbreak\ \rho_{2,12}^{45},\allowbreak\ \rho_{3,12}^{25},\allowbreak\ \rho_{5,9}^{15},\allowbreak\ \rho_{5,9}^{45},\allowbreak\ \rho_{1,5}^{45},\allowbreak\ \rho_{1,5}^{34},\allowbreak\ \rho_{5,9}^{34},\allowbreak\ \rho_{5,9}^{13},\allowbreak\ \rho_{5,12}^{23},\allowbreak\ \rho_{5,10}^{12},\allowbreak\ \rho_{5,11}^{13},\allowbreak\ \rho_{5,12}^{13}$ & $[b_{11},b_{2}]\quad\mathsf{L}$ \\
\end{longtable}\endgroup
\paragraph{Resolve $b_{17}$.} Its halves are $(7,4)$, $(6,3)$. The companion is $[b_{16},a_{6,4}]$.
\begingroup
\footnotesize\setlength{\tabcolsep}{3pt}\renewcommand{\arraystretch}{1.23}
\begin{longtable}{@{}p{.195\textwidth}p{.55\textwidth}p{.225\textwidth}@{}}
\caption{Companion certificate for $b_{17}$.}\\
\hline Target & Successive rectangles & Terminal and ground \\\hline
\endfirsthead
\multicolumn{3}{l}{\small\itshape Continued: Companion certificate for $b_{17}$.}\\
\hline Target & Successive rectangles & Terminal and ground \\\hline
\endhead
\hline\endfoot
1. $[a_{6,4},b_{16}]$ & $\text{none}$ & $[a_{6,4},b_{16}]\quad\mathsf{L}$ \\
\end{longtable}\endgroup
\paragraph{Resolve $b_{18}$.} Its halves are $(6,1)$, $(4,3)$. The companion is $[b_{17},a_{4,1}]$.
\begingroup
\footnotesize\setlength{\tabcolsep}{3pt}\renewcommand{\arraystretch}{1.23}
\begin{longtable}{@{}p{.195\textwidth}p{.55\textwidth}p{.225\textwidth}@{}}
\caption{Companion certificate for $b_{18}$.}\\
\hline Target & Successive rectangles & Terminal and ground \\\hline
\endfirsthead
\multicolumn{3}{l}{\small\itshape Continued: Companion certificate for $b_{18}$.}\\
\hline Target & Successive rectangles & Terminal and ground \\\hline
\endhead
\hline\endfoot
1. $[a_{4,1},b_{17}]$ & $\rho_{4,7}^{14}$ & $[b_{15},b_{7}]\quad\mathsf{L}$ \\
\end{longtable}\endgroup
\paragraph{Resolve $b_{19}$.} Its halves are $(6,5)$, $(3,3)$. The companion is $[b_{17},b_{1}]$.
\begingroup
\footnotesize\setlength{\tabcolsep}{3pt}\renewcommand{\arraystretch}{1.23}
\begin{longtable}{@{}p{.195\textwidth}p{.55\textwidth}p{.225\textwidth}@{}}
\caption{Companion certificate for $b_{19}$.}\\
\hline Target & Successive rectangles & Terminal and ground \\\hline
\endfirsthead
\multicolumn{3}{l}{\small\itshape Continued: Companion certificate for $b_{19}$.}\\
\hline Target & Successive rectangles & Terminal and ground \\\hline
\endhead
\hline\endfoot
1. $[b_{1},b_{17}]$ & $\rho_{5,7}^{45},\allowbreak\ \rho_{5,7}^{15},\allowbreak\ \rho_{3,9}^{25},\allowbreak\ \rho_{2,9}^{45},\allowbreak\ \rho_{8,9}^{24},\allowbreak\ \rho_{7,8}^{14},\allowbreak\ \rho_{6,8}^{13}$ & $[a_{8,3},b_{18}]\quad\mathsf{L}$ \\
\end{longtable}\endgroup

\section{Complete seed orthogonality derivations}\label{app:seedorth}
All nineteen two-edges are now identified. In this appendix the full seed
labels in~\eqref{eq:seed} are used, and row 12 is literally row 12.
Of the 171 different two-edge pairs, the 115 omitted from the following table
share a row or a column in their supports. Each of the other 56 pairs is
listed. The terminal $\mathsf I$ cases are prescribed values one, expressed
as zero defects; they are not orthogonality assertions about a vector and itself.
\begingroup
\footnotesize\setlength{\tabcolsep}{3pt}\renewcommand{\arraystretch}{1.23}
\begin{longtable}{@{}p{.195\textwidth}p{.55\textwidth}p{.225\textwidth}@{}}
\caption{The 56 non-line two-edge pairs of the seed.}\\
\hline Target & Successive rectangles & Terminal and ground \\\hline
\endfirsthead
\multicolumn{3}{l}{\small\itshape Continued: The 56 non-line two-edge pairs of the seed.}\\
\hline Target & Successive rectangles & Terminal and ground \\\hline
\endhead
\hline\endfoot
B1. $[b_{1},b_{3}]$ & $\rho_{1,3}^{45},\allowbreak\ \rho_{1,3}^{34},\allowbreak\ \rho_{1,6}^{45},\allowbreak\ \rho_{1,6}^{34}$ & $[b_{17},b_{3}]\quad\mathsf{L}$ \\
B2. $[b_{1},b_{7}]$ & $\rho_{3,4}^{25},\allowbreak\ \rho_{2,4}^{45}$ & $[b_{6},b_{7}]\quad\mathsf{L}$ \\
B3. $[b_{1},b_{8}]$ & $\rho_{3,11}^{35}$ & $[b_{12},b_{19}]\quad\mathsf{L}$ \\
B4. $[b_{1},b_{9}]$ & $\rho_{3,10}^{35}$ & $[a_{10,5},b_{19}]\quad\mathsf{L}$ \\
B5. $[b_{1},b_{10}]$ & $\rho_{3,11}^{15}$ & $[a_{3,1},b_{12}]\quad\mathsf{L}$ \\
B6. $[b_{1},b_{11}]$ & $\rho_{3,12}^{35}$ & $[W,b_{19}]\quad\mathsf{L}$ \\
B7. $[b_{1},b_{14}]$ & $\rho_{5,9}^{15},\allowbreak\ \rho_{5,9}^{45},\allowbreak\ \rho_{1,3}^{45},\allowbreak\ \rho_{1,3}^{34},\allowbreak\ \rho_{1,6}^{45},\allowbreak\ \rho_{1,6}^{34}$ & $[b_{17},b_{3}]\quad\mathsf{L}$ \\
B8. $[b_{1},b_{15}]$ & $\rho_{5,7}^{15},\allowbreak\ \rho_{5,7}^{45},\allowbreak\ \rho_{3,6}^{35}$ & $[b_{19},b_{19}]\quad\mathsf{I}$ \\
B9. $[b_{1},b_{16}]$ & $\rho_{3,7}^{35}$ & $[a_{7,5},b_{19}]\quad\mathsf{L}$ \\
B10. $[b_{1},b_{17}]$ & $\rho_{3,6}^{35}$ & $[b_{19},b_{19}]\quad\mathsf{I}$ \\
B11. $[b_{1},b_{18}]$ & $\rho_{3,6}^{15}$ & $[a_{3,1},b_{19}]\quad\mathsf{L}$ \\
B12. $[b_{2},b_{4}]$ & $\rho_{1,5}^{15},\allowbreak\ \rho_{1,3}^{15},\allowbreak\ \rho_{1,3}^{13},\allowbreak\ \rho_{1,6}^{15}$ & $[b_{18},b_{4}]\quad\mathsf{L}$ \\
B13. $[b_{2},b_{5}]$ & $\rho_{2,5}^{24}$ & $[a_{5,2},b_{13}]\quad\mathsf{L}$ \\
B14. $[b_{2},b_{6}]$ & $\rho_{2,5}^{45}$ & $[b_{1},b_{13}]\quad\mathsf{L}$ \\
B15. $[b_{2},b_{8}]$ & $\rho_{5,11}^{13},\allowbreak\ \rho_{5,12}^{13}$ & $[b_{11},b_{2}]\quad\mathsf{L}$ \\
B16. $[b_{2},b_{9}]$ & $\rho_{5,10}^{13}$ & $[a_{5,3},b_{11}]\quad\mathsf{L}$ \\
B17. $[b_{2},b_{19}]$ & $\rho_{5,6}^{15},\allowbreak\ \rho_{3,6}^{15}$ & $[a_{3,1},b_{19}]\quad\mathsf{L}$ \\
B18. $[b_{3},b_{5}]$ & $\rho_{1,2}^{24}$ & $[a_{1,2},b_{13}]\quad\mathsf{L}$ \\
B19. $[b_{3},b_{6}]$ & $\rho_{8,9}^{24},\allowbreak\ \rho_{7,8}^{14},\allowbreak\ \rho_{6,11}^{35}$ & $[b_{19},b_{8}]\quad\mathsf{L}$ \\
B20. $[b_{3},b_{8}]$ & $\rho_{9,10}^{24},\allowbreak\ \rho_{7,10}^{14}$ & $[b_{11},b_{17}]\quad\mathsf{L}$ \\
B21. $[b_{3},b_{9}]$ & $\rho_{9,11}^{24},\allowbreak\ \rho_{7,11}^{14},\allowbreak\ \rho_{6,12}^{13}$ & $[b_{11},b_{18}]\quad\mathsf{L}$ \\
B22. $[b_{3},b_{10}]$ & $\rho_{9,12}^{14}$ & $[b_{14},b_{16}]\quad\mathsf{L}$ \\
B23. $[b_{3},b_{11}]$ & $\rho_{1,12}^{34}$ & $[b_{16},b_{4}]\quad\mathsf{L}$ \\
B24. $[b_{3},b_{12}]$ & $\rho_{1,11}^{45},\allowbreak\ \rho_{1,11}^{34},\allowbreak\ \rho_{9,10}^{24},\allowbreak\ \rho_{7,10}^{14}$ & $[b_{11},b_{17}]\quad\mathsf{L}$ \\
B25. $[b_{3},b_{18}]$ & $\rho_{4,9}^{34},\allowbreak\ \rho_{4,9}^{23}$ & $[b_{15},b_{18}]\quad\mathsf{L}$ \\
B26. $[b_{3},b_{19}]$ & $\rho_{1,6}^{45},\allowbreak\ \rho_{1,6}^{34}$ & $[b_{17},b_{3}]\quad\mathsf{L}$ \\
B27. $[b_{4},b_{7}]$ & $\rho_{1,4}^{34},\allowbreak\ \rho_{4,9}^{34},\allowbreak\ \rho_{4,9}^{23}$ & $[b_{15},b_{18}]\quad\mathsf{L}$ \\
B28. $[b_{4},b_{10}]$ & $\rho_{1,11}^{15}$ & $[a_{1,1},b_{12}]\quad\mathsf{L}$ \\
B29. $[b_{4},b_{13}]$ & $\rho_{1,2}^{45},\allowbreak\ \rho_{8,9}^{24},\allowbreak\ \rho_{7,8}^{14},\allowbreak\ \rho_{6,11}^{35}$ & $[b_{19},b_{8}]\quad\mathsf{L}$ \\
B30. $[b_{4},b_{14}]$ & $\rho_{1,12}^{23},\allowbreak\ \rho_{1,10}^{12},\allowbreak\ \rho_{1,11}^{13},\allowbreak\ \rho_{1,11}^{15}$ & $[a_{1,1},b_{12}]\quad\mathsf{L}$ \\
B31. $[b_{4},b_{15}]$ & $\rho_{1,7}^{13},\allowbreak\ \rho_{1,12}^{14},\allowbreak\ \rho_{9,12}^{14}$ & $[b_{14},b_{16}]\quad\mathsf{L}$ \\
B32. $[b_{5},b_{10}]$ & $\rho_{8,11}^{15}$ & $[b_{12},b_{12}]\quad\mathsf{I}$ \\
B33. $[b_{5},b_{11}]$ & $\rho_{8,10}^{15}$ & $[a_{10,5},b_{12}]\quad\mathsf{L}$ \\
B34. $[b_{5},b_{16}]$ & $\rho_{2,12}^{24}$ & $[b_{13},b_{14}]\quad\mathsf{L}$ \\
B35. $[b_{5},b_{17}]$ & $\rho_{2,7}^{24}$ & $[a_{7,2},b_{13}]\quad\mathsf{L}$ \\
B36. $[b_{5},b_{18}]$ & $\rho_{6,8}^{15}$ & $[b_{12},b_{19}]\quad\mathsf{L}$ \\
B37. $[b_{6},b_{10}]$ & $\rho_{2,11}^{15}$ & $[a_{2,1},b_{12}]\quad\mathsf{L}$ \\
B38. $[b_{6},b_{11}]$ & $\rho_{8,10}^{12}$ & $[b_{12},b_{8}]\quad\mathsf{L}$ \\
B39. $[b_{6},b_{16}]$ & $\rho_{8,12}^{24},\allowbreak\ \rho_{8,9}^{14},\allowbreak\ \rho_{1,11}^{45},\allowbreak\ \rho_{1,11}^{34},\allowbreak\ \rho_{9,10}^{24},\allowbreak\ \rho_{7,10}^{14}$ & $[b_{11},b_{17}]\quad\mathsf{L}$ \\
B40. $[b_{6},b_{17}]$ & $\rho_{2,6}^{35}$ & $[a_{2,3},b_{19}]\quad\mathsf{L}$ \\
B41. $[b_{6},b_{18}]$ & $\rho_{6,8}^{12},\allowbreak\ \rho_{6,11}^{25}$ & $[b_{19},b_{9}]\quad\mathsf{L}$ \\
B42. $[b_{7},b_{10}]$ & $\rho_{4,12}^{12}$ & $[a_{4,1},b_{14}]\quad\mathsf{L}$ \\
B43. $[b_{7},b_{11}]$ & $\rho_{4,12}^{34}$ & $[b_{16},b_{18}]\quad\mathsf{L}$ \\
B44. $[b_{7},b_{12}]$ & $\rho_{4,11}^{25},\allowbreak\ \rho_{4,10}^{35},\allowbreak\ \rho_{6,10}^{15}$ & $[b_{11},b_{19}]\quad\mathsf{L}$ \\
B45. $[b_{7},b_{19}]$ & $\rho_{3,4}^{23},\allowbreak\ \rho_{2,4}^{34},\allowbreak\ \rho_{2,4}^{23},\allowbreak\ \rho_{6,8}^{15}$ & $[b_{12},b_{19}]\quad\mathsf{L}$ \\
B46. $[b_{10},b_{13}]$ & $\rho_{3,12}^{12}$ & $[a_{3,1},b_{14}]\quad\mathsf{L}$ \\
B47. $[b_{10},b_{17}]$ & $\rho_{6,12}^{13}$ & $[b_{11},b_{18}]\quad\mathsf{L}$ \\
B48. $[b_{10},b_{19}]$ & $\rho_{3,12}^{13}$ & $[a_{3,1},b_{11}]\quad\mathsf{L}$ \\
B49. $[b_{11},b_{13}]$ & $\rho_{2,12}^{34}$ & $[a_{2,3},b_{16}]\quad\mathsf{L}$ \\
B50. $[b_{12},b_{13}]$ & $\rho_{3,11}^{25},\allowbreak\ \rho_{3,10}^{35}$ & $[a_{10,5},b_{19}]\quad\mathsf{L}$ \\
B51. $[b_{12},b_{16}]$ & $\rho_{8,12}^{14},\allowbreak\ \rho_{8,11}^{14}$ & $[a_{11,4},b_{12}]\quad\mathsf{L}$ \\
B52. $[b_{12},b_{17}]$ & $\rho_{6,11}^{35}$ & $[b_{19},b_{8}]\quad\mathsf{L}$ \\
B53. $[b_{13},b_{18}]$ & $\rho_{2,4}^{34},\allowbreak\ \rho_{2,4}^{23},\allowbreak\ \rho_{6,8}^{15}$ & $[b_{12},b_{19}]\quad\mathsf{L}$ \\
B54. $[b_{14},b_{17}]$ & $\rho_{7,9}^{14}$ & $[b_{15},b_{3}]\quad\mathsf{L}$ \\
B55. $[b_{14},b_{19}]$ & $\rho_{3,12}^{23},\allowbreak\ \rho_{2,12}^{34}$ & $[a_{2,3},b_{16}]\quad\mathsf{L}$ \\
B56. $[b_{15},b_{19}]$ & $\rho_{3,9}^{23},\allowbreak\ \rho_{2,9}^{34},\allowbreak\ \rho_{1,2}^{34},\allowbreak\ \rho_{1,2}^{45},\allowbreak\ \rho_{8,9}^{24},\allowbreak\ \rho_{7,8}^{14},\allowbreak\ \rho_{6,11}^{35}$ & $[b_{19},b_{8}]\quad\mathsf{L}$ \\
\end{longtable}\endgroup

\subsection{One-edge/two-edge pairs}
There are 418 targets. Exactly 210 pairs share a row or column and need no
rectangle. The following two tables list all other targets, first the 147
one-transfer cases and then the 61 longer cases. A ground marked $\mathsf B$
is supplied by the preceding two-edge table or its line-grounded complement.
Thus the tables have no circular dependency. In particular, the complement
of these 208 listed targets is determined directly by the displayed support grid.
\begingroup
\footnotesize\setlength{\tabcolsep}{3pt}\renewcommand{\arraystretch}{1.23}
\begin{longtable}{@{}p{.195\textwidth}p{.55\textwidth}p{.225\textwidth}@{}}
\caption{The 147 one-transfer mixed pairs.}\\
\hline Target & Successive rectangles & Terminal and ground \\\hline
\endfirsthead
\multicolumn{3}{l}{\small\itshape Continued: The 147 one-transfer mixed pairs.}\\
\hline Target & Successive rectangles & Terminal and ground \\\hline
\endhead
\hline\endfoot
M1. $[a_{1,1},b_{1}]$ & $\rho_{1,5}^{15}$ & $[b_{2},b_{4}]\quad\mathsf{B}$ \\
M2. $[a_{1,1},b_{5}]$ & $\rho_{1,8}^{15}$ & $[b_{12},b_{4}]\quad\mathsf{B}$ \\
M3. $[a_{1,1},b_{8}]$ & $\rho_{1,11}^{13}$ & $[b_{10},b_{4}]\quad\mathsf{B}$ \\
M4. $[a_{1,1},b_{9}]$ & $\rho_{1,10}^{13}$ & $[b_{11},b_{4}]\quad\mathsf{B}$ \\
M5. $[a_{1,1},b_{16}]$ & $\rho_{1,12}^{14}$ & $[b_{10},b_{3}]\quad\mathsf{B}$ \\
M6. $[a_{1,1},b_{17}]$ & $\rho_{1,6}^{13}$ & $[b_{18},b_{4}]\quad\mathsf{B}$ \\
M7. $[a_{1,1},b_{19}]$ & $\rho_{1,6}^{15}$ & $[b_{18},b_{4}]\quad\mathsf{B}$ \\
M8. $[a_{1,2},b_{1}]$ & $\rho_{1,3}^{25}$ & $[b_{13},b_{4}]\quad\mathsf{B}$ \\
M9. $[a_{1,2},b_{10}]$ & $\rho_{1,12}^{12}$ & $[a_{1,1},b_{14}]\quad\mathsf{L}$ \\
M10. $[a_{1,2},b_{11}]$ & $\rho_{1,12}^{23}$ & $[b_{14},b_{4}]\quad\mathsf{B}$ \\
M11. $[a_{1,2},b_{12}]$ & $\rho_{1,11}^{25}$ & $[b_{4},b_{9}]\quad\mathsf{B}$ \\
M12. $[a_{1,2},b_{16}]$ & $\rho_{1,12}^{24}$ & $[b_{14},b_{3}]\quad\mathsf{B}$ \\
M13. $[a_{1,2},b_{18}]$ & $\rho_{1,4}^{23}$ & $[b_{4},b_{7}]\quad\mathsf{B}$ \\
M14. $[a_{1,2},b_{19}]$ & $\rho_{1,3}^{23}$ & $[b_{13},b_{4}]\quad\mathsf{B}$ \\
M15. $[a_{2,1},b_{1}]$ & $\rho_{2,5}^{15}$ & $[b_{2},b_{6}]\quad\mathsf{B}$ \\
M16. $[a_{2,1},b_{3}]$ & $\rho_{2,9}^{14}$ & $[b_{13},b_{14}]\quad\mathsf{B}$ \\
M17. $[a_{2,1},b_{8}]$ & $\rho_{2,10}^{12}$ & $[b_{11},b_{5}]\quad\mathsf{B}$ \\
M18. $[a_{2,1},b_{9}]$ & $\rho_{2,10}^{13}$ & $[a_{2,3},b_{11}]\quad\mathsf{L}$ \\
M19. $[a_{2,1},b_{16}]$ & $\rho_{2,12}^{14}$ & $[b_{10},b_{13}]\quad\mathsf{B}$ \\
M20. $[a_{2,1},b_{17}]$ & $\rho_{2,6}^{13}$ & $[a_{2,3},b_{18}]\quad\mathsf{L}$ \\
M21. $[a_{2,1},b_{19}]$ & $\rho_{2,6}^{15}$ & $[b_{18},b_{6}]\quad\mathsf{B}$ \\
M22. $[a_{2,3},b_{1}]$ & $\rho_{2,3}^{35}$ & $[b_{19},b_{6}]\quad\mathsf{B}$ \\
M23. $[a_{2,3},b_{3}]$ & $\rho_{1,2}^{34}$ & $[b_{13},b_{4}]\quad\mathsf{B}$ \\
M24. $[a_{2,3},b_{7}]$ & $\rho_{2,4}^{23}$ & $[b_{18},b_{5}]\quad\mathsf{B}$ \\
M25. $[a_{2,3},b_{10}]$ & $\rho_{2,12}^{13}$ & $[a_{2,1},b_{11}]\quad\mathsf{L}$ \\
M26. $[a_{2,3},b_{12}]$ & $\rho_{2,11}^{35}$ & $[b_{6},b_{8}]\quad\mathsf{B}$ \\
M27. $[a_{2,3},b_{14}]$ & $\rho_{2,12}^{23}$ & $[b_{11},b_{5}]\quad\mathsf{B}$ \\
M28. $[a_{3,1},b_{5}]$ & $\rho_{2,3}^{12}$ & $[a_{2,1},b_{13}]\quad\mathsf{L}$ \\
M29. $[a_{3,1},b_{6}]$ & $\rho_{3,8}^{12}$ & $[b_{12},b_{13}]\quad\mathsf{B}$ \\
M30. $[a_{3,1},b_{8}]$ & $\rho_{3,11}^{13}$ & $[b_{10},b_{19}]\quad\mathsf{B}$ \\
M31. $[a_{3,1},b_{9}]$ & $\rho_{3,10}^{13}$ & $[b_{11},b_{19}]\quad\mathsf{B}$ \\
M32. $[a_{3,1},b_{16}]$ & $\rho_{3,7}^{13}$ & $[b_{15},b_{19}]\quad\mathsf{B}$ \\
M33. $[a_{3,1},b_{17}]$ & $\rho_{3,6}^{13}$ & $[b_{18},b_{19}]\quad\mathsf{B}$ \\
M34. $[a_{3,4},b_{4}]$ & $\rho_{1,3}^{45}$ & $[b_{1},b_{3}]\quad\mathsf{B}$ \\
M35. $[a_{3,4},b_{5}]$ & $\rho_{2,3}^{24}$ & $[b_{13},b_{13}]\quad\mathsf{I}$ \\
M36. $[a_{3,4},b_{6}]$ & $\rho_{2,3}^{45}$ & $[b_{1},b_{13}]\quad\mathsf{B}$ \\
M37. $[a_{3,4},b_{8}]$ & $\rho_{3,10}^{24}$ & $[a_{10,4},b_{13}]\quad\mathsf{L}$ \\
M38. $[a_{3,4},b_{9}]$ & $\rho_{3,11}^{24}$ & $[a_{11,4},b_{13}]\quad\mathsf{L}$ \\
M39. $[a_{3,4},b_{11}]$ & $\rho_{3,12}^{34}$ & $[b_{16},b_{19}]\quad\mathsf{B}$ \\
M40. $[a_{3,4},b_{14}]$ & $\rho_{3,12}^{24}$ & $[b_{13},b_{16}]\quad\mathsf{B}$ \\
M41. $[a_{3,4},b_{15}]$ & $\rho_{3,9}^{24}$ & $[b_{13},b_{3}]\quad\mathsf{B}$ \\
M42. $[a_{3,4},b_{18}]$ & $\rho_{3,4}^{34}$ & $[b_{19},b_{7}]\quad\mathsf{B}$ \\
M43. $[a_{4,1},b_{3}]$ & $\rho_{4,9}^{14}$ & $[b_{14},b_{7}]\quad\mathsf{B}$ \\
M44. $[a_{4,1},b_{4}]$ & $\rho_{1,4}^{13}$ & $[a_{1,1},b_{18}]\quad\mathsf{L}$ \\
M45. $[a_{4,1},b_{5}]$ & $\rho_{4,8}^{15}$ & $[a_{4,5},b_{12}]\quad\mathsf{L}$ \\
M46. $[a_{4,1},b_{6}]$ & $\rho_{4,8}^{12}$ & $[b_{12},b_{7}]\quad\mathsf{B}$ \\
M47. $[a_{4,1},b_{8}]$ & $\rho_{4,11}^{13}$ & $[b_{10},b_{18}]\quad\mathsf{B}$ \\
M48. $[a_{4,1},b_{9}]$ & $\rho_{4,10}^{13}$ & $[b_{11},b_{18}]\quad\mathsf{B}$ \\
M49. $[a_{4,1},b_{16}]$ & $\rho_{4,7}^{13}$ & $[b_{15},b_{18}]\quad\mathsf{B}$ \\
M50. $[a_{4,1},b_{17}]$ & $\rho_{4,6}^{13}$ & $[b_{18},b_{18}]\quad\mathsf{I}$ \\
M51. $[a_{4,1},b_{19}]$ & $\rho_{3,4}^{13}$ & $[a_{3,1},b_{18}]\quad\mathsf{L}$ \\
M52. $[a_{4,5},b_{2}]$ & $\rho_{4,5}^{45}$ & $[b_{1},b_{7}]\quad\mathsf{B}$ \\
M53. $[a_{4,5},b_{3}]$ & $\rho_{1,4}^{45}$ & $[b_{4},b_{7}]\quad\mathsf{B}$ \\
M54. $[a_{4,5},b_{8}]$ & $\rho_{4,11}^{35}$ & $[b_{12},b_{18}]\quad\mathsf{B}$ \\
M55. $[a_{4,5},b_{9}]$ & $\rho_{4,11}^{25}$ & $[b_{12},b_{7}]\quad\mathsf{B}$ \\
M56. $[a_{4,5},b_{10}]$ & $\rho_{4,11}^{15}$ & $[a_{4,1},b_{12}]\quad\mathsf{L}$ \\
M57. $[a_{4,5},b_{13}]$ & $\rho_{2,4}^{45}$ & $[b_{6},b_{7}]\quad\mathsf{B}$ \\
M58. $[a_{4,5},b_{17}]$ & $\rho_{4,6}^{35}$ & $[b_{18},b_{19}]\quad\mathsf{B}$ \\
M59. $[a_{5,2},b_{3}]$ & $\rho_{5,9}^{24}$ & $[b_{15},b_{2}]\quad\mathsf{B}$ \\
M60. $[a_{5,2},b_{10}]$ & $\rho_{5,12}^{12}$ & $[b_{14},b_{2}]\quad\mathsf{B}$ \\
M61. $[a_{5,2},b_{11}]$ & $\rho_{5,10}^{12}$ & $[b_{2},b_{8}]\quad\mathsf{B}$ \\
M62. $[a_{5,2},b_{12}]$ & $\rho_{5,11}^{25}$ & $[b_{1},b_{9}]\quad\mathsf{B}$ \\
M63. $[a_{5,2},b_{16}]$ & $\rho_{5,12}^{24}$ & $[b_{14},b_{2}]\quad\mathsf{B}$ \\
M64. $[a_{5,3},b_{3}]$ & $\rho_{1,5}^{34}$ & $[b_{2},b_{4}]\quad\mathsf{B}$ \\
M65. $[a_{5,3},b_{7}]$ & $\rho_{4,5}^{34}$ & $[b_{18},b_{2}]\quad\mathsf{B}$ \\
M66. $[a_{5,3},b_{10}]$ & $\rho_{5,12}^{13}$ & $[b_{11},b_{2}]\quad\mathsf{B}$ \\
M67. $[a_{5,3},b_{12}]$ & $\rho_{5,11}^{35}$ & $[b_{1},b_{8}]\quad\mathsf{B}$ \\
M68. $[a_{5,3},b_{15}]$ & $\rho_{5,7}^{13}$ & $[b_{16},b_{2}]\quad\mathsf{B}$ \\
M69. $[a_{6,2},b_{1}]$ & $\rho_{3,6}^{25}$ & $[b_{13},b_{19}]\quad\mathsf{B}$ \\
M70. $[a_{6,2},b_{10}]$ & $\rho_{6,12}^{12}$ & $[b_{14},b_{18}]\quad\mathsf{B}$ \\
M71. $[a_{6,2},b_{11}]$ & $\rho_{6,10}^{12}$ & $[b_{18},b_{8}]\quad\mathsf{B}$ \\
M72. $[a_{6,2},b_{12}]$ & $\rho_{6,11}^{25}$ & $[b_{19},b_{9}]\quad\mathsf{B}$ \\
M73. $[a_{6,2},b_{16}]$ & $\rho_{6,7}^{23}$ & $[a_{7,2},b_{17}]\quad\mathsf{L}$ \\
M74. $[a_{6,4},b_{1}]$ & $\rho_{3,6}^{45}$ & $[a_{3,4},b_{19}]\quad\mathsf{L}$ \\
M75. $[a_{6,4},b_{4}]$ & $\rho_{1,6}^{34}$ & $[b_{17},b_{3}]\quad\mathsf{B}$ \\
M76. $[a_{6,4},b_{5}]$ & $\rho_{2,6}^{24}$ & $[a_{6,2},b_{13}]\quad\mathsf{L}$ \\
M77. $[a_{6,4},b_{6}]$ & $\rho_{2,6}^{45}$ & $[b_{13},b_{19}]\quad\mathsf{B}$ \\
M78. $[a_{6,4},b_{8}]$ & $\rho_{6,11}^{34}$ & $[a_{11,4},b_{17}]\quad\mathsf{L}$ \\
M79. $[a_{6,4},b_{9}]$ & $\rho_{6,10}^{34}$ & $[a_{10,4},b_{17}]\quad\mathsf{L}$ \\
M80. $[a_{6,4},b_{10}]$ & $\rho_{6,12}^{14}$ & $[b_{16},b_{18}]\quad\mathsf{B}$ \\
M81. $[a_{6,4},b_{11}]$ & $\rho_{6,12}^{34}$ & $[b_{16},b_{17}]\quad\mathsf{B}$ \\
M82. $[a_{6,4},b_{14}]$ & $\rho_{6,9}^{14}$ & $[b_{18},b_{3}]\quad\mathsf{B}$ \\
M83. $[a_{6,4},b_{15}]$ & $\rho_{6,7}^{14}$ & $[b_{17},b_{18}]\quad\mathsf{B}$ \\
M84. $[a_{7,2},b_{2}]$ & $\rho_{5,7}^{12}$ & $[a_{5,2},b_{15}]\quad\mathsf{L}$ \\
M85. $[a_{7,2},b_{3}]$ & $\rho_{7,9}^{24}$ & $[b_{15},b_{17}]\quad\mathsf{B}$ \\
M86. $[a_{7,2},b_{10}]$ & $\rho_{7,12}^{12}$ & $[b_{14},b_{15}]\quad\mathsf{B}$ \\
M87. $[a_{7,2},b_{11}]$ & $\rho_{7,10}^{12}$ & $[b_{15},b_{8}]\quad\mathsf{B}$ \\
M88. $[a_{7,2},b_{12}]$ & $\rho_{7,8}^{12}$ & $[b_{15},b_{6}]\quad\mathsf{B}$ \\
M89. $[a_{7,2},b_{18}]$ & $\rho_{4,7}^{23}$ & $[b_{16},b_{7}]\quad\mathsf{B}$ \\
M90. $[a_{7,2},b_{19}]$ & $\rho_{3,7}^{23}$ & $[b_{13},b_{16}]\quad\mathsf{B}$ \\
M91. $[a_{7,5},b_{2}]$ & $\rho_{5,7}^{15}$ & $[b_{1},b_{15}]\quad\mathsf{B}$ \\
M92. $[a_{7,5},b_{3}]$ & $\rho_{1,7}^{45}$ & $[b_{17},b_{4}]\quad\mathsf{B}$ \\
M93. $[a_{7,5},b_{8}]$ & $\rho_{7,11}^{35}$ & $[b_{12},b_{16}]\quad\mathsf{B}$ \\
M94. $[a_{7,5},b_{10}]$ & $\rho_{7,11}^{15}$ & $[b_{12},b_{15}]\quad\mathsf{B}$ \\
M95. $[a_{7,5},b_{11}]$ & $\rho_{7,12}^{35}$ & $[W,b_{16}]\quad\mathsf{L}$ \\
M96. $[a_{7,5},b_{13}]$ & $\rho_{2,7}^{45}$ & $[b_{17},b_{6}]\quad\mathsf{B}$ \\
M97. $[a_{7,5},b_{14}]$ & $\rho_{7,9}^{15}$ & $[a_{9,5},b_{15}]\quad\mathsf{L}$ \\
M98. $[a_{7,5},b_{18}]$ & $\rho_{6,7}^{15}$ & $[b_{15},b_{19}]\quad\mathsf{B}$ \\
M99. $[a_{8,3},b_{1}]$ & $\rho_{3,8}^{35}$ & $[b_{19},b_{5}]\quad\mathsf{B}$ \\
M100. $[a_{8,3},b_{7}]$ & $\rho_{4,8}^{23}$ & $[b_{18},b_{6}]\quad\mathsf{B}$ \\
M101. $[a_{8,3},b_{10}]$ & $\rho_{8,12}^{13}$ & $[b_{11},b_{12}]\quad\mathsf{B}$ \\
M102. $[a_{8,3},b_{13}]$ & $\rho_{3,8}^{23}$ & $[b_{19},b_{6}]\quad\mathsf{B}$ \\
M103. $[a_{8,3},b_{14}]$ & $\rho_{8,12}^{23}$ & $[b_{11},b_{6}]\quad\mathsf{B}$ \\
M104. $[a_{8,3},b_{15}]$ & $\rho_{7,8}^{13}$ & $[b_{12},b_{16}]\quad\mathsf{B}$ \\
M105. $[a_{8,4},b_{1}]$ & $\rho_{5,8}^{45}$ & $[b_{2},b_{5}]\quad\mathsf{B}$ \\
M106. $[a_{8,4},b_{4}]$ & $\rho_{1,8}^{45}$ & $[b_{3},b_{5}]\quad\mathsf{B}$ \\
M107. $[a_{8,4},b_{10}]$ & $\rho_{8,11}^{14}$ & $[a_{11,4},b_{12}]\quad\mathsf{L}$ \\
M108. $[a_{8,4},b_{11}]$ & $\rho_{8,12}^{34}$ & $[a_{8,3},b_{16}]\quad\mathsf{L}$ \\
M109. $[a_{8,4},b_{14}]$ & $\rho_{8,9}^{14}$ & $[b_{12},b_{3}]\quad\mathsf{B}$ \\
M110. $[a_{8,4},b_{15}]$ & $\rho_{8,9}^{24}$ & $[b_{3},b_{6}]\quad\mathsf{B}$ \\
M111. $[a_{9,3},b_{1}]$ & $\rho_{3,9}^{35}$ & $[a_{9,5},b_{19}]\quad\mathsf{L}$ \\
M112. $[a_{9,3},b_{7}]$ & $\rho_{4,9}^{23}$ & $[b_{15},b_{18}]\quad\mathsf{B}$ \\
M113. $[a_{9,3},b_{10}]$ & $\rho_{9,12}^{13}$ & $[b_{11},b_{14}]\quad\mathsf{B}$ \\
M114. $[a_{9,3},b_{13}]$ & $\rho_{3,9}^{23}$ & $[b_{15},b_{19}]\quad\mathsf{B}$ \\
M115. $[a_{9,5},b_{2}]$ & $\rho_{5,9}^{45}$ & $[b_{1},b_{3}]\quad\mathsf{B}$ \\
M116. $[a_{9,5},b_{9}]$ & $\rho_{9,11}^{25}$ & $[b_{12},b_{15}]\quad\mathsf{B}$ \\
M117. $[a_{9,5},b_{10}]$ & $\rho_{9,11}^{15}$ & $[b_{12},b_{14}]\quad\mathsf{B}$ \\
M118. $[a_{9,5},b_{13}]$ & $\rho_{3,9}^{25}$ & $[b_{1},b_{15}]\quad\mathsf{B}$ \\
M119. $[a_{9,5},b_{17}]$ & $\rho_{6,9}^{35}$ & $[a_{9,3},b_{19}]\quad\mathsf{L}$ \\
M120. $[a_{9,5},b_{18}]$ & $\rho_{6,9}^{15}$ & $[b_{14},b_{19}]\quad\mathsf{B}$ \\
M121. $[a_{10,4},b_{4}]$ & $\rho_{1,10}^{34}$ & $[b_{3},b_{9}]\quad\mathsf{B}$ \\
M122. $[a_{10,4},b_{5}]$ & $\rho_{2,10}^{24}$ & $[b_{13},b_{8}]\quad\mathsf{B}$ \\
M123. $[a_{10,4},b_{10}]$ & $\rho_{10,12}^{14}$ & $[b_{11},b_{16}]\quad\mathsf{B}$ \\
M124. $[a_{10,4},b_{14}]$ & $\rho_{10,12}^{24}$ & $[b_{16},b_{8}]\quad\mathsf{B}$ \\
M125. $[a_{10,4},b_{15}]$ & $\rho_{7,10}^{14}$ & $[b_{11},b_{17}]\quad\mathsf{B}$ \\
M126. $[a_{10,4},b_{18}]$ & $\rho_{4,10}^{34}$ & $[b_{7},b_{9}]\quad\mathsf{B}$ \\
M127. $[a_{10,5},b_{2}]$ & $\rho_{5,10}^{15}$ & $[b_{1},b_{11}]\quad\mathsf{B}$ \\
M128. $[a_{10,5},b_{10}]$ & $\rho_{10,11}^{15}$ & $[b_{11},b_{12}]\quad\mathsf{B}$ \\
M129. $[a_{10,5},b_{13}]$ & $\rho_{3,10}^{25}$ & $[b_{1},b_{8}]\quad\mathsf{B}$ \\
M130. $[a_{10,5},b_{17}]$ & $\rho_{6,10}^{35}$ & $[b_{19},b_{9}]\quad\mathsf{B}$ \\
M131. $[a_{10,5},b_{18}]$ & $\rho_{6,10}^{15}$ & $[b_{11},b_{19}]\quad\mathsf{B}$ \\
M132. $[a_{11,4},b_{1}]$ & $\rho_{5,11}^{45}$ & $[b_{12},b_{2}]\quad\mathsf{B}$ \\
M133. $[a_{11,4},b_{4}]$ & $\rho_{1,11}^{34}$ & $[b_{3},b_{8}]\quad\mathsf{B}$ \\
M134. $[a_{11,4},b_{5}]$ & $\rho_{2,11}^{24}$ & $[b_{13},b_{9}]\quad\mathsf{B}$ \\
M135. $[a_{11,4},b_{6}]$ & $\rho_{2,11}^{45}$ & $[b_{12},b_{13}]\quad\mathsf{B}$ \\
M136. $[a_{11,4},b_{11}]$ & $\rho_{11,12}^{34}$ & $[b_{16},b_{8}]\quad\mathsf{B}$ \\
M137. $[a_{11,4},b_{14}]$ & $\rho_{11,12}^{24}$ & $[b_{16},b_{9}]\quad\mathsf{B}$ \\
M138. $[a_{11,4},b_{15}]$ & $\rho_{9,11}^{24}$ & $[b_{3},b_{9}]\quad\mathsf{B}$ \\
M139. $[a_{11,4},b_{18}]$ & $\rho_{4,11}^{34}$ & $[b_{7},b_{8}]\quad\mathsf{B}$ \\
M140. $[W,b_{2}]$ & $\rho_{5,12}^{15}$ & $[b_{1},b_{10}]\quad\mathsf{B}$ \\
M141. $[W,b_{3}]$ & $\rho_{1,12}^{45}$ & $[b_{16},b_{4}]\quad\mathsf{B}$ \\
M142. $[W,b_{8}]$ & $\rho_{11,12}^{35}$ & $[b_{11},b_{12}]\quad\mathsf{B}$ \\
M143. $[W,b_{9}]$ & $\rho_{11,12}^{25}$ & $[b_{12},b_{14}]\quad\mathsf{B}$ \\
M144. $[W,b_{13}]$ & $\rho_{3,12}^{25}$ & $[b_{1},b_{14}]\quad\mathsf{B}$ \\
M145. $[W,b_{15}]$ & $\rho_{9,12}^{25}$ & $[a_{9,5},b_{14}]\quad\mathsf{L}$ \\
M146. $[W,b_{17}]$ & $\rho_{6,12}^{35}$ & $[b_{11},b_{19}]\quad\mathsf{B}$ \\
M147. $[W,b_{18}]$ & $\rho_{6,12}^{15}$ & $[b_{10},b_{19}]\quad\mathsf{B}$ \\
\end{longtable}\endgroup
\begingroup
\footnotesize\setlength{\tabcolsep}{3pt}\renewcommand{\arraystretch}{1.23}
\begin{longtable}{@{}p{.195\textwidth}p{.55\textwidth}p{.225\textwidth}@{}}
\caption{The 61 remaining mixed pairs.}\\
\hline Target & Successive rectangles & Terminal and ground \\\hline
\endfirsthead
\multicolumn{3}{l}{\small\itshape Continued: The 61 remaining mixed pairs.}\\
\hline Target & Successive rectangles & Terminal and ground \\\hline
\endhead
\hline\endfoot
N1. $[a_{1,1},b_{6}]$ & $\rho_{1,8}^{12},\allowbreak\ \rho_{1,11}^{25}$ & $[b_{4},b_{9}]\quad\mathsf{B}$ \\
N2. $[a_{1,1},b_{7}]$ & $\rho_{1,4}^{14},\allowbreak\ \rho_{4,9}^{14}$ & $[b_{14},b_{7}]\quad\mathsf{B}$ \\
N3. $[a_{1,1},b_{13}]$ & $\rho_{1,2}^{14},\allowbreak\ \rho_{2,9}^{14}$ & $[b_{13},b_{14}]\quad\mathsf{B}$ \\
N4. $[a_{1,2},b_{2}]$ & $\rho_{1,5}^{24},\allowbreak\ \rho_{5,9}^{24}$ & $[b_{15},b_{2}]\quad\mathsf{B}$ \\
N5. $[a_{1,2},b_{17}]$ & $\rho_{1,7}^{24},\allowbreak\ \rho_{7,9}^{24}$ & $[b_{15},b_{17}]\quad\mathsf{B}$ \\
N6. $[a_{2,1},b_{4}]$ & $\rho_{1,2}^{15},\allowbreak\ \rho_{1,8}^{12},\allowbreak\ \rho_{1,11}^{25}$ & $[b_{4},b_{9}]\quad\mathsf{B}$ \\
N7. $[a_{2,1},b_{7}]$ & $\rho_{2,4}^{12},\allowbreak\ \rho_{4,8}^{15}$ & $[a_{4,5},b_{12}]\quad\mathsf{L}$ \\
N8. $[a_{2,3},b_{2}]$ & $\rho_{2,5}^{34},\allowbreak\ \rho_{3,5}^{23},\allowbreak\ \rho_{5,6}^{25},\allowbreak\ \rho_{3,6}^{25}$ & $[b_{13},b_{19}]\quad\mathsf{B}$ \\
N9. $[a_{2,3},b_{15}]$ & $\rho_{2,7}^{13},\allowbreak\ \rho_{2,12}^{14}$ & $[b_{10},b_{13}]\quad\mathsf{B}$ \\
N10. $[a_{3,1},b_{3}]$ & $\rho_{3,9}^{14},\allowbreak\ \rho_{3,12}^{24}$ & $[b_{13},b_{16}]\quad\mathsf{B}$ \\
N11. $[a_{3,1},b_{4}]$ & $\rho_{1,3}^{13},\allowbreak\ \rho_{1,6}^{15}$ & $[b_{18},b_{4}]\quad\mathsf{B}$ \\
N12. $[a_{3,1},b_{7}]$ & $\rho_{3,4}^{12},\allowbreak\ \rho_{2,4}^{14},\allowbreak\ \rho_{2,4}^{12},\allowbreak\ \rho_{4,8}^{15}$ & $[a_{4,5},b_{12}]\quad\mathsf{L}$ \\
N13. $[a_{3,4},b_{10}]$ & $\rho_{3,12}^{14},\allowbreak\ \rho_{3,7}^{13}$ & $[b_{15},b_{19}]\quad\mathsf{B}$ \\
N14. $[a_{3,4},b_{12}]$ & $\rho_{3,11}^{45},\allowbreak\ \rho_{5,11}^{45}$ & $[b_{12},b_{2}]\quad\mathsf{B}$ \\
N15. $[a_{4,1},b_{1}]$ & $\rho_{4,5}^{15},\allowbreak\ \rho_{4,5}^{45}$ & $[b_{1},b_{7}]\quad\mathsf{B}$ \\
N16. $[a_{4,1},b_{13}]$ & $\rho_{2,4}^{14},\allowbreak\ \rho_{2,4}^{12},\allowbreak\ \rho_{4,8}^{15}$ & $[a_{4,5},b_{12}]\quad\mathsf{L}$ \\
N17. $[a_{4,5},b_{11}]$ & $\rho_{4,12}^{35},\allowbreak\ \rho_{6,12}^{15}$ & $[b_{10},b_{19}]\quad\mathsf{B}$ \\
N18. $[a_{4,5},b_{14}]$ & $\rho_{4,12}^{25},\allowbreak\ \rho_{4,12}^{45},\allowbreak\ \rho_{4,7}^{35},\allowbreak\ \rho_{6,7}^{15}$ & $[b_{15},b_{19}]\quad\mathsf{B}$ \\
N19. $[a_{4,5},b_{15}]$ & $\rho_{4,9}^{25},\allowbreak\ \rho_{4,9}^{45},\allowbreak\ \rho_{1,4}^{45}$ & $[b_{4},b_{7}]\quad\mathsf{B}$ \\
N20. $[a_{4,5},b_{16}]$ & $\rho_{4,7}^{35},\allowbreak\ \rho_{6,7}^{15}$ & $[b_{15},b_{19}]\quad\mathsf{B}$ \\
N21. $[a_{5,2},b_{4}]$ & $\rho_{1,5}^{25},\allowbreak\ \rho_{1,3}^{25}$ & $[b_{13},b_{4}]\quad\mathsf{B}$ \\
N22. $[a_{5,2},b_{17}]$ & $\rho_{5,7}^{24},\allowbreak\ \rho_{5,7}^{12}$ & $[a_{5,2},b_{15}]\quad\mathsf{L}$ \\
N23. $[a_{5,2},b_{18}]$ & $\rho_{4,5}^{23},\allowbreak\ \rho_{4,5}^{34}$ & $[b_{18},b_{2}]\quad\mathsf{B}$ \\
N24. $[a_{5,2},b_{19}]$ & $\rho_{5,6}^{25},\allowbreak\ \rho_{3,6}^{25}$ & $[b_{13},b_{19}]\quad\mathsf{B}$ \\
N25. $[a_{5,3},b_{5}]$ & $\rho_{5,8}^{35},\allowbreak\ \rho_{3,8}^{35}$ & $[b_{19},b_{5}]\quad\mathsf{B}$ \\
N26. $[a_{5,3},b_{6}]$ & $\rho_{2,5}^{35},\allowbreak\ \rho_{2,3}^{35}$ & $[b_{19},b_{6}]\quad\mathsf{B}$ \\
N27. $[a_{5,3},b_{13}]$ & $\rho_{3,5}^{23},\allowbreak\ \rho_{5,6}^{25},\allowbreak\ \rho_{3,6}^{25}$ & $[b_{13},b_{19}]\quad\mathsf{B}$ \\
N28. $[a_{5,3},b_{14}]$ & $\rho_{5,12}^{23},\allowbreak\ \rho_{5,10}^{12}$ & $[b_{2},b_{8}]\quad\mathsf{B}$ \\
N29. $[a_{6,2},b_{2}]$ & $\rho_{5,6}^{12},\allowbreak\ \rho_{4,5}^{23},\allowbreak\ \rho_{4,5}^{34}$ & $[b_{18},b_{2}]\quad\mathsf{B}$ \\
N30. $[a_{6,2},b_{3}]$ & $\rho_{6,9}^{24},\allowbreak\ \rho_{6,7}^{14}$ & $[b_{17},b_{18}]\quad\mathsf{B}$ \\
N31. $[a_{6,2},b_{4}]$ & $\rho_{1,6}^{25},\allowbreak\ \rho_{1,3}^{23}$ & $[b_{13},b_{4}]\quad\mathsf{B}$ \\
N32. $[a_{6,4},b_{12}]$ & $\rho_{6,11}^{45},\allowbreak\ \rho_{3,11}^{34},\allowbreak\ \rho_{3,10}^{24}$ & $[a_{10,4},b_{13}]\quad\mathsf{L}$ \\
N33. $[a_{7,2},b_{1}]$ & $\rho_{3,7}^{25},\allowbreak\ \rho_{2,7}^{45}$ & $[b_{17},b_{6}]\quad\mathsf{B}$ \\
N34. $[a_{7,2},b_{4}]$ & $\rho_{1,7}^{23},\allowbreak\ \rho_{1,12}^{24}$ & $[b_{14},b_{3}]\quad\mathsf{B}$ \\
N35. $[a_{7,5},b_{7}]$ & $\rho_{4,7}^{45},\allowbreak\ \rho_{4,6}^{35}$ & $[b_{18},b_{19}]\quad\mathsf{B}$ \\
N36. $[a_{7,5},b_{9}]$ & $\rho_{7,11}^{25},\allowbreak\ \rho_{7,8}^{12}$ & $[b_{15},b_{6}]\quad\mathsf{B}$ \\
N37. $[a_{8,3},b_{2}]$ & $\rho_{5,8}^{13},\allowbreak\ \rho_{5,11}^{35}$ & $[b_{1},b_{8}]\quad\mathsf{B}$ \\
N38. $[a_{8,3},b_{3}]$ & $\rho_{1,8}^{34},\allowbreak\ \rho_{1,8}^{45}$ & $[b_{3},b_{5}]\quad\mathsf{B}$ \\
N39. $[a_{8,4},b_{8}]$ & $\rho_{8,10}^{24},\allowbreak\ \rho_{2,10}^{45},\allowbreak\ \rho_{3,10}^{25}$ & $[b_{1},b_{8}]\quad\mathsf{B}$ \\
N40. $[a_{8,4},b_{9}]$ & $\rho_{8,11}^{24},\allowbreak\ \rho_{2,11}^{45}$ & $[b_{12},b_{13}]\quad\mathsf{B}$ \\
N41. $[a_{8,4},b_{18}]$ & $\rho_{4,8}^{34},\allowbreak\ \rho_{4,8}^{23}$ & $[b_{18},b_{6}]\quad\mathsf{B}$ \\
N42. $[a_{8,4},b_{19}]$ & $\rho_{6,8}^{45},\allowbreak\ \rho_{2,6}^{24}$ & $[a_{6,2},b_{13}]\quad\mathsf{L}$ \\
N43. $[a_{9,3},b_{2}]$ & $\rho_{5,9}^{34},\allowbreak\ \rho_{1,5}^{34}$ & $[b_{2},b_{4}]\quad\mathsf{B}$ \\
N44. $[a_{9,3},b_{5}]$ & $\rho_{2,9}^{23},\allowbreak\ \rho_{2,7}^{13},\allowbreak\ \rho_{2,12}^{14}$ & $[b_{10},b_{13}]\quad\mathsf{B}$ \\
N45. $[a_{9,3},b_{6}]$ & $\rho_{8,9}^{23},\allowbreak\ \rho_{7,8}^{13}$ & $[b_{12},b_{16}]\quad\mathsf{B}$ \\
N46. $[a_{9,3},b_{12}]$ & $\rho_{8,9}^{13},\allowbreak\ \rho_{8,12}^{23}$ & $[b_{11},b_{6}]\quad\mathsf{B}$ \\
N47. $[a_{9,5},b_{7}]$ & $\rho_{4,9}^{45},\allowbreak\ \rho_{1,4}^{45}$ & $[b_{4},b_{7}]\quad\mathsf{B}$ \\
N48. $[a_{9,5},b_{8}]$ & $\rho_{9,10}^{25},\allowbreak\ \rho_{7,10}^{15},\allowbreak\ \rho_{7,12}^{35}$ & $[W,b_{16}]\quad\mathsf{L}$ \\
N49. $[a_{9,5},b_{11}]$ & $\rho_{9,10}^{15},\allowbreak\ \rho_{10,12}^{25},\allowbreak\ \rho_{11,12}^{35}$ & $[b_{11},b_{12}]\quad\mathsf{B}$ \\
N50. $[a_{9,5},b_{16}]$ & $\rho_{9,12}^{45},\allowbreak\ \rho_{1,12}^{45}$ & $[b_{16},b_{4}]\quad\mathsf{B}$ \\
N51. $[a_{10,4},b_{1}]$ & $\rho_{5,10}^{45},\allowbreak\ \rho_{5,10}^{15}$ & $[b_{1},b_{11}]\quad\mathsf{B}$ \\
N52. $[a_{10,4},b_{6}]$ & $\rho_{2,10}^{45},\allowbreak\ \rho_{3,10}^{25}$ & $[b_{1},b_{8}]\quad\mathsf{B}$ \\
N53. $[a_{10,4},b_{12}]$ & $\rho_{8,10}^{14},\allowbreak\ \rho_{8,12}^{34}$ & $[a_{8,3},b_{16}]\quad\mathsf{L}$ \\
N54. $[a_{10,4},b_{19}]$ & $\rho_{3,10}^{34},\allowbreak\ \rho_{3,11}^{24}$ & $[a_{11,4},b_{13}]\quad\mathsf{L}$ \\
N55. $[a_{10,5},b_{3}]$ & $\rho_{1,10}^{45},\allowbreak\ \rho_{1,10}^{34}$ & $[b_{3},b_{9}]\quad\mathsf{B}$ \\
N56. $[a_{10,5},b_{7}]$ & $\rho_{4,10}^{25},\allowbreak\ \rho_{4,11}^{35}$ & $[b_{12},b_{18}]\quad\mathsf{B}$ \\
N57. $[a_{10,5},b_{14}]$ & $\rho_{10,12}^{25},\allowbreak\ \rho_{11,12}^{35}$ & $[b_{11},b_{12}]\quad\mathsf{B}$ \\
N58. $[a_{10,5},b_{15}]$ & $\rho_{7,10}^{15},\allowbreak\ \rho_{7,12}^{35}$ & $[W,b_{16}]\quad\mathsf{L}$ \\
N59. $[a_{10,5},b_{16}]$ & $\rho_{7,10}^{35},\allowbreak\ \rho_{7,11}^{25},\allowbreak\ \rho_{7,8}^{12}$ & $[b_{15},b_{6}]\quad\mathsf{B}$ \\
N60. $[a_{11,4},b_{19}]$ & $\rho_{3,11}^{34},\allowbreak\ \rho_{3,10}^{24}$ & $[a_{10,4},b_{13}]\quad\mathsf{L}$ \\
N61. $[W,b_{7}]$ & $\rho_{4,12}^{45},\allowbreak\ \rho_{4,7}^{35},\allowbreak\ \rho_{6,7}^{15}$ & $[b_{15},b_{19}]\quad\mathsf{B}$ \\
\end{longtable}\endgroup

\subsection{One-edge/one-edge pairs}
The final 231 targets are obtained as in Proposition~\ref{prop:seed}.
Their companion cannot contain two one-edges unless a forbidden one-edge
rectangle occurs. All other companion values have just been certified.
This explains coverage without adding an unnecessary 231-line table.

\clearpage
\section{Complete symbolic interface derivations}\label{app:interface}
In all four tables, old row 12 means the \emph{terminal row $w$}.
The unchanged old one-edge labels are $a_{r,c}$, with $W=a_{12,5}$.
The old part of the template is the first eleven seed rows with
$b_{14}$ at $(9,1)$ replaced by $P_0$ and $b_{16}$ at $(7,3)$ by $Q_0$,
followed by $(b_{10},P_t,b_{11},Q_t,W)$ at row $w$.

The allowed common-column grounds use the full supports
$C(P_j)=\{1,2\}$, $C(Q_j)=\{3,4\}$ and $C(U_j)=\{5\}$.
All old-label supports are read from this grid; the path grounds are precisely
Lemma~\ref{lem:path} and~\eqref{eq:Panchors}--\eqref{eq:Qanchors}.
No interface hypothesis is used except the explicit $\mathsf E$ premise
in the prefix table.

Each target consists of a displayed chain label and one of the 43 old labels.
All targets not printed in the applicable table satisfy one of these grounds;
this is a direct support comparison with the fixed grid. The printed targets
are the complete complementary lists. If symbolic labels coincide at an
endpoint, interpret brackets using their actual identity after substitution.

\subsection{Interior: $2\le i\le t-1$}\label{tab:interior}
Set $x=r_{i-1}$, $y=r_i$, $z=r_{i+1}$. The only extra rows used are
\[
\begin{array}{c|ccccc}
x&P_{i-1}&P_{i-2}&Q_{i-1}&Q_{i-2}&U_{i-1}\\
y&P_i&P_{i-1}&Q_i&Q_{i-1}&U_i\\
z&P_{i+1}&P_i&Q_{i+1}&Q_i&U_{i+1}
\end{array}.
\]
There are 86 targets: 57 grounded and 29 in the table. The smallest index is
$i-2\ge0$; at $i=2$ its first-half representative is the old port cell.
The largest row index is $i+1\le t$. No row outside the construction is used.
\begingroup
\footnotesize\setlength{\tabcolsep}{3pt}\renewcommand{\arraystretch}{1.23}
\begin{longtable}{@{}p{.195\textwidth}p{.55\textwidth}p{.225\textwidth}@{}}
\caption{Interior old-label relations at an arbitrary index.}\\
\hline Target & Successive rectangles & Terminal and ground \\\hline
\endfirsthead
\multicolumn{3}{l}{\small\itshape Continued: Interior old-label relations at an arbitrary index.}\\
\hline Target & Successive rectangles & Terminal and ground \\\hline
\endhead
\hline\endfoot
I1. $[P_{i},a_{10,4}]$ & $\rho_{10,y}^{14}$ & $[Q_{i-1},b_{11}]\quad\mathsf{C}$ \\
I2. $[P_{i},a_{10,5}]$ & $\rho_{10,y}^{15},\allowbreak\ \rho_{w,y}^{35}$ & $[Q_{i},W]\quad\mathsf{H}$ \\
I3. $[P_{i},a_{11,4}]$ & $\rho_{11,z}^{24}$ & $[Q_{i},b_{9}]\quad\mathsf{C}$ \\
I4. $[P_{i},a_{2,3}]$ & $\rho_{2,y}^{13},\allowbreak\ \rho_{2,z}^{14}$ & $[P_{i+1},b_{13}]\quad\mathsf{C}$ \\
I5. $[P_{i},a_{3,4}]$ & $\rho_{3,z}^{24}$ & $[Q_{i},b_{13}]\quad\mathsf{C}$ \\
I6. $[P_{i},a_{4,5}]$ & $\rho_{4,z}^{25},\allowbreak\ \rho_{4,z}^{45},\allowbreak\ \rho_{4,y}^{35},\allowbreak\ \rho_{6,y}^{15},\allowbreak\ \rho_{3,z}^{23}$ & $[Q_{i+1},b_{13}]\quad\mathsf{C}$ \\
I7. $[P_{i},a_{5,3}]$ & $\rho_{5,y}^{13}$ & $[Q_{i},b_{2}]\quad\mathsf{C}$ \\
I8. $[P_{i},a_{6,4}]$ & $\rho_{6,y}^{14}$ & $[Q_{i-1},b_{18}]\quad\mathsf{C}$ \\
I9. $[P_{i},a_{7,5}]$ & $\rho_{7,y}^{15},\allowbreak\ \rho_{9,y}^{25}$ & $[P_{i-1},a_{9,5}]\quad\mathsf{H}$ \\
I10. $[P_{i},a_{8,3}]$ & $\rho_{8,y}^{13},\allowbreak\ \rho_{11,y}^{35},\allowbreak\ \rho_{10,y}^{25},\allowbreak\ \rho_{10,x}^{15},\allowbreak\ \rho_{w,x}^{35}$ & $[Q_{i-1},W]\quad\mathsf{H}$ \\
I11. $[P_{i},a_{8,4}]$ & $\rho_{8,z}^{24},\allowbreak\ \rho_{2,z}^{45},\allowbreak\ \rho_{3,z}^{25},\allowbreak\ \rho_{5,y}^{15},\allowbreak\ \rho_{5,y}^{45},\allowbreak\ \rho_{3,x}^{35}$ & $[U_{i-1},b_{19}]\quad\mathsf{C}$ \\
I12. $[P_{i},b_{1}]$ & $\rho_{5,y}^{15},\allowbreak\ \rho_{5,y}^{45},\allowbreak\ \rho_{3,x}^{35}$ & $[U_{i-1},b_{19}]\quad\mathsf{C}$ \\
I13. $[P_{i},b_{17}]$ & $\rho_{6,y}^{13}$ & $[Q_{i},b_{18}]\quad\mathsf{C}$ \\
I14. $[P_{i},b_{19}]$ & $\rho_{3,z}^{23}$ & $[Q_{i+1},b_{13}]\quad\mathsf{C}$ \\
I15. $[P_{i},b_{4}]$ & $\rho_{1,y}^{13},\allowbreak\ \rho_{1,z}^{14}$ & $[P_{i+1},b_{3}]\quad\mathsf{H}$ \\
I16. $[Q_{i},a_{10,5}]$ & $\rho_{10,y}^{35},\allowbreak\ \rho_{11,y}^{25}$ & $[P_{i-1},b_{12}]\quad\mathsf{C}$ \\
I17. $[Q_{i},a_{1,1}]$ & $\rho_{1,z}^{14}$ & $[P_{i+1},b_{3}]\quad\mathsf{H}$ \\
I18. $[Q_{i},a_{1,2}]$ & $\rho_{1,z}^{24}$ & $[P_{i},b_{3}]\quad\mathsf{H}$ \\
I19. $[Q_{i},a_{2,1}]$ & $\rho_{2,z}^{14}$ & $[P_{i+1},b_{13}]\quad\mathsf{C}$ \\
I20. $[Q_{i},a_{3,1}]$ & $\rho_{3,y}^{13},\allowbreak\ \rho_{3,z}^{23}$ & $[Q_{i+1},b_{13}]\quad\mathsf{C}$ \\
I21. $[Q_{i},a_{4,1}]$ & $\rho_{4,y}^{13}$ & $[P_{i},b_{18}]\quad\mathsf{C}$ \\
I22. $[Q_{i},a_{4,5}]$ & $\rho_{4,y}^{35},\allowbreak\ \rho_{6,y}^{15},\allowbreak\ \rho_{3,z}^{23}$ & $[Q_{i+1},b_{13}]\quad\mathsf{C}$ \\
I23. $[Q_{i},a_{5,2}]$ & $\rho_{5,z}^{24}$ & $[P_{i},b_{2}]\quad\mathsf{C}$ \\
I24. $[Q_{i},a_{6,2}]$ & $\rho_{6,y}^{23},\allowbreak\ \rho_{6,x}^{13}$ & $[Q_{i-1},b_{18}]\quad\mathsf{C}$ \\
I25. $[Q_{i},a_{9,5}]$ & $\rho_{9,z}^{45},\allowbreak\ \rho_{1,z}^{45}$ & $[Q_{i},b_{4}]\quad\mathsf{C}$ \\
I26. $[Q_{i},b_{1}]$ & $\rho_{3,y}^{35}$ & $[U_{i},b_{19}]\quad\mathsf{C}$ \\
I27. $[Q_{i},b_{12}]$ & $\rho_{11,y}^{35},\allowbreak\ \rho_{10,y}^{25},\allowbreak\ \rho_{10,x}^{15},\allowbreak\ \rho_{w,x}^{35}$ & $[Q_{i-1},W]\quad\mathsf{H}$ \\
I28. $[Q_{i},b_{5}]$ & $\rho_{2,z}^{24}$ & $[P_{i},b_{13}]\quad\mathsf{C}$ \\
I29. $[Q_{i},b_{6}]$ & $\rho_{2,z}^{45},\allowbreak\ \rho_{3,z}^{25},\allowbreak\ \rho_{5,y}^{15},\allowbreak\ \rho_{5,y}^{45},\allowbreak\ \rho_{3,x}^{35}$ & $[U_{i-1},b_{19}]\quad\mathsf{C}$ \\
\end{longtable}\endgroup

\subsection{Terminal index: $t\ge2$}\label{tab:tail}
Set $x=r_{t-1}$ and $y=r_t$. The extra rows are
\[
\begin{array}{c|ccccc}
x&P_{t-1}&P_{t-2}&Q_{t-1}&Q_{t-2}&U_{t-1}\\
y&P_t&P_{t-1}&Q_t&Q_{t-1}&U_t
\end{array}.
\]
There are 86 targets: 57 grounded and the following 29. No label $P_{t+1}$
or $Q_{t+1}$ and no nonexistent extra row is introduced.
\begingroup
\footnotesize\setlength{\tabcolsep}{3pt}\renewcommand{\arraystretch}{1.23}
\begin{longtable}{@{}p{.195\textwidth}p{.55\textwidth}p{.225\textwidth}@{}}
\caption{Terminal old-label relations.}\\
\hline Target & Successive rectangles & Terminal and ground \\\hline
\endfirsthead
\multicolumn{3}{l}{\small\itshape Continued: Terminal old-label relations.}\\
\hline Target & Successive rectangles & Terminal and ground \\\hline
\endhead
\hline\endfoot
T1. $[P_{t},a_{10,4}]$ & $\rho_{10,y}^{14}$ & $[Q_{t-1},b_{11}]\quad\mathsf{C}$ \\
T2. $[P_{t},a_{10,5}]$ & $\rho_{10,w}^{25},\allowbreak\ \rho_{11,w}^{35}$ & $[b_{11},b_{12}]\quad\mathsf{C}$ \\
T3. $[P_{t},a_{11,4}]$ & $\rho_{11,w}^{24}$ & $[Q_{t},b_{9}]\quad\mathsf{C}$ \\
T4. $[P_{t},a_{2,3}]$ & $\rho_{2,w}^{23},\allowbreak\ \rho_{8,10}^{15}$ & $[a_{10,5},b_{12}]\quad\mathsf{C}$ \\
T5. $[P_{t},a_{3,4}]$ & $\rho_{3,w}^{24}$ & $[Q_{t},b_{13}]\quad\mathsf{C}$ \\
T6. $[P_{t},a_{4,5}]$ & $\rho_{4,w}^{25},\allowbreak\ \rho_{4,w}^{45},\allowbreak\ \rho_{4,y}^{35},\allowbreak\ \rho_{6,y}^{15},\allowbreak\ \rho_{3,w}^{23},\allowbreak\ \rho_{2,w}^{34}$ & $[Q_{t},a_{2,3}]\quad\mathsf{C}$ \\
T7. $[P_{t},a_{5,3}]$ & $\rho_{5,y}^{13}$ & $[Q_{t},b_{2}]\quad\mathsf{C}$ \\
T8. $[P_{t},a_{6,4}]$ & $\rho_{6,y}^{14}$ & $[Q_{t-1},b_{18}]\quad\mathsf{C}$ \\
T9. $[P_{t},a_{7,5}]$ & $\rho_{7,y}^{15},\allowbreak\ \rho_{9,y}^{25}$ & $[P_{t-1},a_{9,5}]\quad\mathsf{H}$ \\
T10. $[P_{t},a_{8,3}]$ & $\rho_{8,w}^{23},\allowbreak\ \rho_{8,10}^{12}$ & $[b_{12},b_{8}]\quad\mathsf{L}$ \\
T11. $[P_{t},a_{8,4}]$ & $\rho_{8,w}^{24},\allowbreak\ \rho_{2,w}^{45},\allowbreak\ \rho_{3,w}^{25},\allowbreak\ \rho_{5,y}^{15},\allowbreak\ \rho_{5,y}^{45},\allowbreak\ \rho_{3,x}^{35}$ & $[U_{t-1},b_{19}]\quad\mathsf{C}$ \\
T12. $[P_{t},b_{1}]$ & $\rho_{5,y}^{15},\allowbreak\ \rho_{5,y}^{45},\allowbreak\ \rho_{3,x}^{35}$ & $[U_{t-1},b_{19}]\quad\mathsf{C}$ \\
T13. $[P_{t},b_{17}]$ & $\rho_{6,y}^{13}$ & $[Q_{t},b_{18}]\quad\mathsf{C}$ \\
T14. $[P_{t},b_{19}]$ & $\rho_{3,w}^{23},\allowbreak\ \rho_{2,w}^{34}$ & $[Q_{t},a_{2,3}]\quad\mathsf{C}$ \\
T15. $[P_{t},b_{4}]$ & $\rho_{1,y}^{13},\allowbreak\ \rho_{1,w}^{14},\allowbreak\ \rho_{9,w}^{14}$ & $[P_0,Q_{t}]\quad\mathsf{H}$ \\
T16. $[Q_{t},a_{10,5}]$ & $\rho_{10,y}^{35},\allowbreak\ \rho_{11,y}^{25}$ & $[P_{t-1},b_{12}]\quad\mathsf{C}$ \\
T17. $[Q_{t},a_{1,1}]$ & $\rho_{1,w}^{14},\allowbreak\ \rho_{9,w}^{14}$ & $[P_0,Q_{t}]\quad\mathsf{H}$ \\
T18. $[Q_{t},a_{1,2}]$ & $\rho_{1,w}^{24}$ & $[P_{t},b_{3}]\quad\mathsf{H}$ \\
T19. $[Q_{t},a_{2,1}]$ & $\rho_{2,w}^{14},\allowbreak\ \rho_{3,w}^{12}$ & $[P_{t},a_{3,1}]\quad\mathsf{C}$ \\
T20. $[Q_{t},a_{3,1}]$ & $\rho_{3,y}^{13},\allowbreak\ \rho_{3,w}^{23},\allowbreak\ \rho_{2,w}^{34}$ & $[Q_{t},a_{2,3}]\quad\mathsf{C}$ \\
T21. $[Q_{t},a_{4,1}]$ & $\rho_{4,y}^{13}$ & $[P_{t},b_{18}]\quad\mathsf{C}$ \\
T22. $[Q_{t},a_{4,5}]$ & $\rho_{4,y}^{35},\allowbreak\ \rho_{6,y}^{15},\allowbreak\ \rho_{3,w}^{23},\allowbreak\ \rho_{2,w}^{34}$ & $[Q_{t},a_{2,3}]\quad\mathsf{C}$ \\
T23. $[Q_{t},a_{5,2}]$ & $\rho_{5,w}^{24}$ & $[P_{t},b_{2}]\quad\mathsf{C}$ \\
T24. $[Q_{t},a_{6,2}]$ & $\rho_{6,y}^{23},\allowbreak\ \rho_{6,x}^{13}$ & $[Q_{t-1},b_{18}]\quad\mathsf{C}$ \\
T25. $[Q_{t},a_{9,5}]$ & $\rho_{9,w}^{45},\allowbreak\ \rho_{1,w}^{45}$ & $[Q_{t},b_{4}]\quad\mathsf{C}$ \\
T26. $[Q_{t},b_{1}]$ & $\rho_{3,y}^{35}$ & $[U_{t},b_{19}]\quad\mathsf{C}$ \\
T27. $[Q_{t},b_{12}]$ & $\rho_{8,w}^{14},\allowbreak\ \rho_{8,11}^{14}$ & $[a_{11,4},b_{12}]\quad\mathsf{L}$ \\
T28. $[Q_{t},b_{5}]$ & $\rho_{2,w}^{24}$ & $[P_{t},b_{13}]\quad\mathsf{C}$ \\
T29. $[Q_{t},b_{6}]$ & $\rho_{2,w}^{45},\allowbreak\ \rho_{3,w}^{25},\allowbreak\ \rho_{5,y}^{15},\allowbreak\ \rho_{5,y}^{45},\allowbreak\ \rho_{3,x}^{35}$ & $[U_{t-1},b_{19}]\quad\mathsf{C}$ \\
\end{longtable}\endgroup

\subsection{Prefix indices 0 and 1: $t\ge2$}\label{tab:prefix}
Set $x=r_1$, $y=r_2$. The extra rows are
\[
\begin{array}{c|ccccc}
x&P_1&P_0&Q_1&Q_0&U_1\\
y&P_2&P_1&Q_2&Q_1&U_2
\end{array}.
\]
The targets use $P_0,Q_0,P_1,Q_1$ and all 43 old labels: 114 are grounded,
and the other 58 are below. The premise $\mathsf E$ is justified before this
table: index 2 is interior if $t\ge3$ and terminal if $t=2$.
\begingroup
\footnotesize\setlength{\tabcolsep}{3pt}\renewcommand{\arraystretch}{1.23}
\begin{longtable}{@{}p{.195\textwidth}p{.55\textwidth}p{.225\textwidth}@{}}
\caption{Prefix old-label relations with the settled index-2 premise.}\\
\hline Target & Successive rectangles & Terminal and ground \\\hline
\endfirsthead
\multicolumn{3}{l}{\small\itshape Continued: Prefix old-label relations with the settled index-2 premise.}\\
\hline Target & Successive rectangles & Terminal and ground \\\hline
\endhead
\hline\endfoot
F1. $[P_{0},a_{10,4}]$ & $\rho_{10,x}^{24}$ & $[Q_{0},b_{8}]\quad\mathsf{C}$ \\
F2. $[P_{0},a_{10,5}]$ & $\rho_{10,x}^{25},\allowbreak\ \rho_{11,x}^{35},\allowbreak\ \rho_{8,y}^{14}$ & $[P_{2},a_{8,4}]\quad\mathsf{E}$ \\
F3. $[P_{0},a_{11,4}]$ & $\rho_{11,x}^{24}$ & $[Q_{0},b_{9}]\quad\mathsf{C}$ \\
F4. $[P_{0},a_{2,3}]$ & $\rho_{2,x}^{23},\allowbreak\ \rho_{2,y}^{24}$ & $[P_{1},b_{13}]\quad\mathsf{C}$ \\
F5. $[P_{0},a_{3,4}]$ & $\rho_{3,x}^{24}$ & $[Q_{0},b_{13}]\quad\mathsf{C}$ \\
F6. $[P_{0},a_{4,5}]$ & $\rho_{4,x}^{25},\allowbreak\ \rho_{4,x}^{45},\allowbreak\ \rho_{4,7}^{35},\allowbreak\ \rho_{6,7}^{15},\allowbreak\ \rho_{3,7}^{13},\allowbreak\ \rho_{3,x}^{14},\allowbreak\ \rho_{3,y}^{24}$ & $[Q_{1},b_{13}]\quad\mathsf{C}$ \\
F7. $[P_{0},a_{5,3}]$ & $\rho_{5,x}^{23},\allowbreak\ \rho_{5,y}^{24}$ & $[P_{1},b_{2}]\quad\mathsf{C}$ \\
F8. $[P_{0},a_{6,4}]$ & $\rho_{6,x}^{24},\allowbreak\ \rho_{6,7}^{23}$ & $[a_{7,2},b_{17}]\quad\mathsf{L}$ \\
F9. $[P_{0},a_{7,5}]$ & $\rho_{7,9}^{15}$ & $[a_{9,5},b_{15}]\quad\mathsf{L}$ \\
F10. $[P_{0},a_{8,3}]$ & $\rho_{8,9}^{13},\allowbreak\ \rho_{9,11}^{35},\allowbreak\ \rho_{9,10}^{25},\allowbreak\ \rho_{7,10}^{15},\allowbreak\ \rho_{7,w}^{35}$ & $[Q_{0},W]\quad\mathsf{H}$ \\
F11. $[P_{0},a_{8,4}]$ & $\rho_{8,9}^{14},\allowbreak\ \rho_{1,11}^{45},\allowbreak\ \rho_{1,11}^{34},\allowbreak\ \rho_{9,10}^{24},\allowbreak\ \rho_{7,10}^{14}$ & $[b_{11},b_{17}]\quad\mathsf{C}$ \\
F12. $[P_{0},b_{1}]$ & $\rho_{5,9}^{15},\allowbreak\ \rho_{5,9}^{45},\allowbreak\ \rho_{1,3}^{45},\allowbreak\ \rho_{1,3}^{34},\allowbreak\ \rho_{1,6}^{45},\allowbreak\ \rho_{1,6}^{34}$ & $[b_{17},b_{3}]\quad\mathsf{C}$ \\
F13. $[P_{0},b_{17}]$ & $\rho_{7,9}^{14}$ & $[b_{15},b_{3}]\quad\mathsf{L}$ \\
F14. $[P_{0},b_{19}]$ & $\rho_{3,x}^{23}$ & $[Q_{1},b_{13}]\quad\mathsf{C}$ \\
F15. $[P_{0},b_{4}]$ & $\rho_{1,x}^{23},\allowbreak\ \rho_{1,y}^{24}$ & $[P_{1},b_{3}]\quad\mathsf{H}$ \\
F16. $[Q_{0},a_{10,5}]$ & $\rho_{7,10}^{35},\allowbreak\ \rho_{7,11}^{25},\allowbreak\ \rho_{7,8}^{12}$ & $[b_{15},b_{6}]\quad\mathsf{C}$ \\
F17. $[Q_{0},a_{1,1}]$ & $\rho_{1,x}^{14}$ & $[P_{1},b_{3}]\quad\mathsf{H}$ \\
F18. $[Q_{0},a_{1,2}]$ & $\rho_{1,x}^{24}$ & $[P_{0},b_{3}]\quad\mathsf{H}$ \\
F19. $[Q_{0},a_{2,1}]$ & $\rho_{2,x}^{14}$ & $[P_{1},b_{13}]\quad\mathsf{C}$ \\
F20. $[Q_{0},a_{3,1}]$ & $\rho_{3,x}^{14},\allowbreak\ \rho_{3,y}^{24}$ & $[Q_{1},b_{13}]\quad\mathsf{C}$ \\
F21. $[Q_{0},a_{4,1}]$ & $\rho_{4,x}^{14}$ & $[P_{1},b_{7}]\quad\mathsf{C}$ \\
F22. $[Q_{0},a_{4,5}]$ & $\rho_{4,7}^{35},\allowbreak\ \rho_{6,7}^{15},\allowbreak\ \rho_{3,7}^{13},\allowbreak\ \rho_{3,x}^{14},\allowbreak\ \rho_{3,y}^{24}$ & $[Q_{1},b_{13}]\quad\mathsf{C}$ \\
F23. $[Q_{0},a_{5,2}]$ & $\rho_{5,x}^{24}$ & $[P_{0},b_{2}]\quad\mathsf{C}$ \\
F24. $[Q_{0},a_{6,2}]$ & $\rho_{6,7}^{23}$ & $[a_{7,2},b_{17}]\quad\mathsf{L}$ \\
F25. $[Q_{0},a_{9,5}]$ & $\rho_{9,x}^{45},\allowbreak\ \rho_{1,x}^{45}$ & $[Q_{0},b_{4}]\quad\mathsf{C}$ \\
F26. $[Q_{0},b_{1}]$ & $\rho_{3,7}^{35}$ & $[a_{7,5},b_{19}]\quad\mathsf{C}$ \\
F27. $[Q_{0},b_{12}]$ & $\rho_{8,x}^{14},\allowbreak\ \rho_{8,y}^{24},\allowbreak\ \rho_{2,y}^{45},\allowbreak\ \rho_{3,y}^{25},\allowbreak\ \rho_{5,x}^{15},\allowbreak\ \rho_{5,x}^{45},\allowbreak\ \rho_{3,7}^{35}$ & $[a_{7,5},b_{19}]\quad\mathsf{C}$ \\
F28. $[Q_{0},b_{5}]$ & $\rho_{2,x}^{24}$ & $[P_{0},b_{13}]\quad\mathsf{C}$ \\
F29. $[Q_{0},b_{6}]$ & $\rho_{8,x}^{24},\allowbreak\ \rho_{8,9}^{14},\allowbreak\ \rho_{1,11}^{45},\allowbreak\ \rho_{1,11}^{34},\allowbreak\ \rho_{9,10}^{24},\allowbreak\ \rho_{7,10}^{14}$ & $[b_{11},b_{17}]\quad\mathsf{C}$ \\
F30. $[P_{1},a_{10,4}]$ & $\rho_{10,x}^{14}$ & $[Q_{0},b_{11}]\quad\mathsf{C}$ \\
F31. $[P_{1},a_{10,5}]$ & $\rho_{10,x}^{15},\allowbreak\ \rho_{w,x}^{35}$ & $[Q_{1},W]\quad\mathsf{H}$ \\
F32. $[P_{1},a_{11,4}]$ & $\rho_{11,y}^{24}$ & $[Q_{1},b_{9}]\quad\mathsf{C}$ \\
F33. $[P_{1},a_{2,3}]$ & $\rho_{2,y}^{23}$ & $[Q_{2},b_{5}]\quad\mathsf{E}$ \\
F34. $[P_{1},a_{3,4}]$ & $\rho_{3,y}^{24}$ & $[Q_{1},b_{13}]\quad\mathsf{C}$ \\
F35. $[P_{1},a_{4,5}]$ & $\rho_{4,y}^{25},\allowbreak\ \rho_{4,y}^{45},\allowbreak\ \rho_{4,x}^{35},\allowbreak\ \rho_{6,x}^{15},\allowbreak\ \rho_{3,y}^{23}$ & $[Q_{2},b_{13}]\quad\mathsf{E}$ \\
F36. $[P_{1},a_{5,3}]$ & $\rho_{5,x}^{13}$ & $[Q_{1},b_{2}]\quad\mathsf{C}$ \\
F37. $[P_{1},a_{6,4}]$ & $\rho_{6,x}^{14}$ & $[Q_{0},b_{18}]\quad\mathsf{C}$ \\
F38. $[P_{1},a_{7,5}]$ & $\rho_{7,x}^{15},\allowbreak\ \rho_{9,x}^{25}$ & $[P_{0},a_{9,5}]\quad\mathsf{H}$ \\
F39. $[P_{1},a_{8,3}]$ & $\rho_{8,y}^{23}$ & $[Q_{2},b_{6}]\quad\mathsf{E}$ \\
F40. $[P_{1},a_{8,4}]$ & $\rho_{8,y}^{24},\allowbreak\ \rho_{2,y}^{45},\allowbreak\ \rho_{3,y}^{25},\allowbreak\ \rho_{5,x}^{15},\allowbreak\ \rho_{5,x}^{45},\allowbreak\ \rho_{3,7}^{35}$ & $[a_{7,5},b_{19}]\quad\mathsf{C}$ \\
F41. $[P_{1},b_{1}]$ & $\rho_{5,x}^{15},\allowbreak\ \rho_{5,x}^{45},\allowbreak\ \rho_{3,7}^{35}$ & $[a_{7,5},b_{19}]\quad\mathsf{C}$ \\
F42. $[P_{1},b_{17}]$ & $\rho_{6,x}^{13}$ & $[Q_{1},b_{18}]\quad\mathsf{C}$ \\
F43. $[P_{1},b_{19}]$ & $\rho_{3,y}^{23}$ & $[Q_{2},b_{13}]\quad\mathsf{E}$ \\
F44. $[P_{1},b_{4}]$ & $\rho_{1,y}^{23}$ & $[Q_{2},a_{1,2}]\quad\mathsf{E}$ \\
F45. $[Q_{1},a_{10,5}]$ & $\rho_{10,x}^{35},\allowbreak\ \rho_{11,x}^{25}$ & $[P_{0},b_{12}]\quad\mathsf{C}$ \\
F46. $[Q_{1},a_{1,1}]$ & $\rho_{1,y}^{14}$ & $[P_{2},b_{3}]\quad\mathsf{E}$ \\
F47. $[Q_{1},a_{1,2}]$ & $\rho_{1,y}^{24}$ & $[P_{1},b_{3}]\quad\mathsf{H}$ \\
F48. $[Q_{1},a_{2,1}]$ & $\rho_{2,y}^{14}$ & $[P_{2},b_{13}]\quad\mathsf{E}$ \\
F49. $[Q_{1},a_{3,1}]$ & $\rho_{3,y}^{14}$ & $[P_{2},a_{3,4}]\quad\mathsf{E}$ \\
F50. $[Q_{1},a_{4,1}]$ & $\rho_{4,x}^{13}$ & $[P_{1},b_{18}]\quad\mathsf{C}$ \\
F51. $[Q_{1},a_{4,5}]$ & $\rho_{4,x}^{35},\allowbreak\ \rho_{6,x}^{15},\allowbreak\ \rho_{3,y}^{23}$ & $[Q_{2},b_{13}]\quad\mathsf{E}$ \\
F52. $[Q_{1},a_{5,2}]$ & $\rho_{5,y}^{24}$ & $[P_{1},b_{2}]\quad\mathsf{C}$ \\
F53. $[Q_{1},a_{6,2}]$ & $\rho_{6,x}^{23},\allowbreak\ \rho_{7,9}^{14}$ & $[b_{15},b_{3}]\quad\mathsf{L}$ \\
F54. $[Q_{1},a_{9,5}]$ & $\rho_{9,y}^{45},\allowbreak\ \rho_{1,y}^{45}$ & $[Q_{1},b_{4}]\quad\mathsf{C}$ \\
F55. $[Q_{1},b_{1}]$ & $\rho_{3,x}^{35}$ & $[U_{1},b_{19}]\quad\mathsf{C}$ \\
F56. $[Q_{1},b_{12}]$ & $\rho_{8,y}^{14}$ & $[P_{2},a_{8,4}]\quad\mathsf{E}$ \\
F57. $[Q_{1},b_{5}]$ & $\rho_{2,y}^{24}$ & $[P_{1},b_{13}]\quad\mathsf{C}$ \\
F58. $[Q_{1},b_{6}]$ & $\rho_{2,y}^{45},\allowbreak\ \rho_{3,y}^{25},\allowbreak\ \rho_{5,x}^{15},\allowbreak\ \rho_{5,x}^{45},\allowbreak\ \rho_{3,7}^{35}$ & $[a_{7,5},b_{19}]\quad\mathsf{C}$ \\
\end{longtable}\endgroup

\subsection{The single inserted row: $t=1$}\label{tab:one}
Here $x=r_1$ and the two nonfixed rows are
\[
\begin{array}{c|ccccc}
x&P_1&P_0&Q_1&Q_0&U_1\\
w&b_{10}&P_1&b_{11}&Q_1&W
\end{array}.
\]
There are 172 targets: 114 grounded and the following 58.
All grounds are immediate column/line/identity relations or instances of the
path lemma. In particular, this table does not assume any index-2 relation.
\begingroup
\footnotesize\setlength{\tabcolsep}{3pt}\renewcommand{\arraystretch}{1.23}
\begin{longtable}{@{}p{.195\textwidth}p{.55\textwidth}p{.225\textwidth}@{}}
\caption{Old-label relations when only one row is inserted.}\\
\hline Target & Successive rectangles & Terminal and ground \\\hline
\endfirsthead
\multicolumn{3}{l}{\small\itshape Continued: Old-label relations when only one row is inserted.}\\
\hline Target & Successive rectangles & Terminal and ground \\\hline
\endhead
\hline\endfoot
O1. $[P_{0},a_{10,4}]$ & $\rho_{10,x}^{24}$ & $[Q_{0},b_{8}]\quad\mathsf{C}$ \\
O2. $[P_{0},a_{10,5}]$ & $\rho_{10,x}^{25},\allowbreak\ \rho_{11,x}^{35},\allowbreak\ \rho_{8,w}^{14},\allowbreak\ \rho_{8,11}^{14}$ & $[a_{11,4},b_{12}]\quad\mathsf{L}$ \\
O3. $[P_{0},a_{11,4}]$ & $\rho_{11,x}^{24}$ & $[Q_{0},b_{9}]\quad\mathsf{C}$ \\
O4. $[P_{0},a_{2,3}]$ & $\rho_{2,x}^{23},\allowbreak\ \rho_{2,w}^{24}$ & $[P_{1},b_{13}]\quad\mathsf{C}$ \\
O5. $[P_{0},a_{3,4}]$ & $\rho_{3,x}^{24}$ & $[Q_{0},b_{13}]\quad\mathsf{C}$ \\
O6. $[P_{0},a_{4,5}]$ & $\rho_{4,x}^{25},\allowbreak\ \rho_{4,x}^{45},\allowbreak\ \rho_{4,7}^{35},\allowbreak\ \rho_{6,7}^{15},\allowbreak\ \rho_{3,7}^{13},\allowbreak\ \rho_{3,x}^{14},\allowbreak\ \rho_{3,w}^{24}$ & $[Q_{1},b_{13}]\quad\mathsf{C}$ \\
O7. $[P_{0},a_{5,3}]$ & $\rho_{5,x}^{23},\allowbreak\ \rho_{5,w}^{24}$ & $[P_{1},b_{2}]\quad\mathsf{C}$ \\
O8. $[P_{0},a_{6,4}]$ & $\rho_{6,x}^{24},\allowbreak\ \rho_{6,7}^{23}$ & $[a_{7,2},b_{17}]\quad\mathsf{L}$ \\
O9. $[P_{0},a_{7,5}]$ & $\rho_{7,9}^{15}$ & $[a_{9,5},b_{15}]\quad\mathsf{L}$ \\
O10. $[P_{0},a_{8,3}]$ & $\rho_{8,9}^{13},\allowbreak\ \rho_{9,11}^{35},\allowbreak\ \rho_{9,10}^{25},\allowbreak\ \rho_{7,10}^{15},\allowbreak\ \rho_{7,w}^{35}$ & $[Q_{0},W]\quad\mathsf{H}$ \\
O11. $[P_{0},a_{8,4}]$ & $\rho_{8,9}^{14},\allowbreak\ \rho_{1,11}^{45},\allowbreak\ \rho_{1,11}^{34},\allowbreak\ \rho_{9,10}^{24},\allowbreak\ \rho_{7,10}^{14}$ & $[b_{11},b_{17}]\quad\mathsf{C}$ \\
O12. $[P_{0},b_{1}]$ & $\rho_{5,9}^{15},\allowbreak\ \rho_{5,9}^{45},\allowbreak\ \rho_{1,3}^{45},\allowbreak\ \rho_{1,3}^{34},\allowbreak\ \rho_{1,6}^{45},\allowbreak\ \rho_{1,6}^{34}$ & $[b_{17},b_{3}]\quad\mathsf{C}$ \\
O13. $[P_{0},b_{17}]$ & $\rho_{7,9}^{14}$ & $[b_{15},b_{3}]\quad\mathsf{L}$ \\
O14. $[P_{0},b_{19}]$ & $\rho_{3,x}^{23}$ & $[Q_{1},b_{13}]\quad\mathsf{C}$ \\
O15. $[P_{0},b_{4}]$ & $\rho_{1,x}^{23},\allowbreak\ \rho_{1,w}^{24}$ & $[P_{1},b_{3}]\quad\mathsf{H}$ \\
O16. $[Q_{0},a_{10,5}]$ & $\rho_{7,10}^{35},\allowbreak\ \rho_{7,11}^{25},\allowbreak\ \rho_{7,8}^{12}$ & $[b_{15},b_{6}]\quad\mathsf{C}$ \\
O17. $[Q_{0},a_{1,1}]$ & $\rho_{1,x}^{14}$ & $[P_{1},b_{3}]\quad\mathsf{H}$ \\
O18. $[Q_{0},a_{1,2}]$ & $\rho_{1,x}^{24}$ & $[P_{0},b_{3}]\quad\mathsf{H}$ \\
O19. $[Q_{0},a_{2,1}]$ & $\rho_{2,x}^{14}$ & $[P_{1},b_{13}]\quad\mathsf{C}$ \\
O20. $[Q_{0},a_{3,1}]$ & $\rho_{3,x}^{14},\allowbreak\ \rho_{3,w}^{24}$ & $[Q_{1},b_{13}]\quad\mathsf{C}$ \\
O21. $[Q_{0},a_{4,1}]$ & $\rho_{4,x}^{14}$ & $[P_{1},b_{7}]\quad\mathsf{C}$ \\
O22. $[Q_{0},a_{4,5}]$ & $\rho_{4,7}^{35},\allowbreak\ \rho_{6,7}^{15},\allowbreak\ \rho_{3,7}^{13},\allowbreak\ \rho_{3,x}^{14},\allowbreak\ \rho_{3,w}^{24}$ & $[Q_{1},b_{13}]\quad\mathsf{C}$ \\
O23. $[Q_{0},a_{5,2}]$ & $\rho_{5,x}^{24}$ & $[P_{0},b_{2}]\quad\mathsf{C}$ \\
O24. $[Q_{0},a_{6,2}]$ & $\rho_{6,7}^{23}$ & $[a_{7,2},b_{17}]\quad\mathsf{L}$ \\
O25. $[Q_{0},a_{9,5}]$ & $\rho_{9,x}^{45},\allowbreak\ \rho_{1,x}^{45}$ & $[Q_{0},b_{4}]\quad\mathsf{C}$ \\
O26. $[Q_{0},b_{1}]$ & $\rho_{3,7}^{35}$ & $[a_{7,5},b_{19}]\quad\mathsf{C}$ \\
O27. $[Q_{0},b_{12}]$ & $\rho_{8,x}^{14},\allowbreak\ \rho_{8,w}^{24},\allowbreak\ \rho_{2,w}^{45},\allowbreak\ \rho_{3,w}^{25},\allowbreak\ \rho_{5,x}^{15},\allowbreak\ \rho_{5,x}^{45},\allowbreak\ \rho_{3,7}^{35}$ & $[a_{7,5},b_{19}]\quad\mathsf{C}$ \\
O28. $[Q_{0},b_{5}]$ & $\rho_{2,x}^{24}$ & $[P_{0},b_{13}]\quad\mathsf{C}$ \\
O29. $[Q_{0},b_{6}]$ & $\rho_{8,x}^{24},\allowbreak\ \rho_{8,9}^{14},\allowbreak\ \rho_{1,11}^{45},\allowbreak\ \rho_{1,11}^{34},\allowbreak\ \rho_{9,10}^{24},\allowbreak\ \rho_{7,10}^{14}$ & $[b_{11},b_{17}]\quad\mathsf{C}$ \\
O30. $[P_{1},a_{10,4}]$ & $\rho_{10,x}^{14}$ & $[Q_{0},b_{11}]\quad\mathsf{C}$ \\
O31. $[P_{1},a_{10,5}]$ & $\rho_{10,w}^{25},\allowbreak\ \rho_{11,w}^{35}$ & $[b_{11},b_{12}]\quad\mathsf{C}$ \\
O32. $[P_{1},a_{11,4}]$ & $\rho_{11,w}^{24}$ & $[Q_{1},b_{9}]\quad\mathsf{C}$ \\
O33. $[P_{1},a_{2,3}]$ & $\rho_{2,w}^{23},\allowbreak\ \rho_{8,10}^{15}$ & $[a_{10,5},b_{12}]\quad\mathsf{C}$ \\
O34. $[P_{1},a_{3,4}]$ & $\rho_{3,w}^{24}$ & $[Q_{1},b_{13}]\quad\mathsf{C}$ \\
O35. $[P_{1},a_{4,5}]$ & $\rho_{4,w}^{25},\allowbreak\ \rho_{4,w}^{45},\allowbreak\ \rho_{4,x}^{35},\allowbreak\ \rho_{6,x}^{15},\allowbreak\ \rho_{3,w}^{23},\allowbreak\ \rho_{2,w}^{34}$ & $[Q_{1},a_{2,3}]\quad\mathsf{C}$ \\
O36. $[P_{1},a_{5,3}]$ & $\rho_{5,x}^{13}$ & $[Q_{1},b_{2}]\quad\mathsf{C}$ \\
O37. $[P_{1},a_{6,4}]$ & $\rho_{6,x}^{14}$ & $[Q_{0},b_{18}]\quad\mathsf{C}$ \\
O38. $[P_{1},a_{7,5}]$ & $\rho_{7,x}^{15},\allowbreak\ \rho_{9,x}^{25}$ & $[P_{0},a_{9,5}]\quad\mathsf{H}$ \\
O39. $[P_{1},a_{8,3}]$ & $\rho_{8,w}^{23},\allowbreak\ \rho_{8,10}^{12}$ & $[b_{12},b_{8}]\quad\mathsf{L}$ \\
O40. $[P_{1},a_{8,4}]$ & $\rho_{8,w}^{24},\allowbreak\ \rho_{2,w}^{45},\allowbreak\ \rho_{3,w}^{25},\allowbreak\ \rho_{5,x}^{15},\allowbreak\ \rho_{5,x}^{45},\allowbreak\ \rho_{3,7}^{35}$ & $[a_{7,5},b_{19}]\quad\mathsf{C}$ \\
O41. $[P_{1},b_{1}]$ & $\rho_{5,x}^{15},\allowbreak\ \rho_{5,x}^{45},\allowbreak\ \rho_{3,7}^{35}$ & $[a_{7,5},b_{19}]\quad\mathsf{C}$ \\
O42. $[P_{1},b_{17}]$ & $\rho_{6,x}^{13}$ & $[Q_{1},b_{18}]\quad\mathsf{C}$ \\
O43. $[P_{1},b_{19}]$ & $\rho_{3,w}^{23},\allowbreak\ \rho_{2,w}^{34}$ & $[Q_{1},a_{2,3}]\quad\mathsf{C}$ \\
O44. $[P_{1},b_{4}]$ & $\rho_{1,x}^{13},\allowbreak\ \rho_{1,w}^{14},\allowbreak\ \rho_{9,w}^{14}$ & $[P_{0},Q_{1}]\quad\mathsf{H}$ \\
O45. $[Q_{1},a_{10,5}]$ & $\rho_{10,x}^{35},\allowbreak\ \rho_{11,x}^{25}$ & $[P_{0},b_{12}]\quad\mathsf{C}$ \\
O46. $[Q_{1},a_{1,1}]$ & $\rho_{1,w}^{14},\allowbreak\ \rho_{9,w}^{14}$ & $[P_{0},Q_{1}]\quad\mathsf{H}$ \\
O47. $[Q_{1},a_{1,2}]$ & $\rho_{1,w}^{24}$ & $[P_{1},b_{3}]\quad\mathsf{H}$ \\
O48. $[Q_{1},a_{2,1}]$ & $\rho_{2,w}^{14},\allowbreak\ \rho_{3,w}^{12}$ & $[P_{1},a_{3,1}]\quad\mathsf{C}$ \\
O49. $[Q_{1},a_{3,1}]$ & $\rho_{3,x}^{13},\allowbreak\ \rho_{3,w}^{23},\allowbreak\ \rho_{2,w}^{34}$ & $[Q_{1},a_{2,3}]\quad\mathsf{C}$ \\
O50. $[Q_{1},a_{4,1}]$ & $\rho_{4,x}^{13}$ & $[P_{1},b_{18}]\quad\mathsf{C}$ \\
O51. $[Q_{1},a_{4,5}]$ & $\rho_{4,x}^{35},\allowbreak\ \rho_{6,x}^{15},\allowbreak\ \rho_{3,w}^{23},\allowbreak\ \rho_{2,w}^{34}$ & $[Q_{1},a_{2,3}]\quad\mathsf{C}$ \\
O52. $[Q_{1},a_{5,2}]$ & $\rho_{5,w}^{24}$ & $[P_{1},b_{2}]\quad\mathsf{C}$ \\
O53. $[Q_{1},a_{6,2}]$ & $\rho_{6,x}^{23},\allowbreak\ \rho_{7,9}^{14}$ & $[b_{15},b_{3}]\quad\mathsf{L}$ \\
O54. $[Q_{1},a_{9,5}]$ & $\rho_{9,w}^{45},\allowbreak\ \rho_{1,w}^{45}$ & $[Q_{1},b_{4}]\quad\mathsf{C}$ \\
O55. $[Q_{1},b_{1}]$ & $\rho_{3,x}^{35}$ & $[U_{1},b_{19}]\quad\mathsf{C}$ \\
O56. $[Q_{1},b_{12}]$ & $\rho_{8,w}^{14},\allowbreak\ \rho_{8,11}^{14}$ & $[a_{11,4},b_{12}]\quad\mathsf{L}$ \\
O57. $[Q_{1},b_{5}]$ & $\rho_{2,w}^{24}$ & $[P_{1},b_{13}]\quad\mathsf{C}$ \\
O58. $[Q_{1},b_{6}]$ & $\rho_{2,w}^{45},\allowbreak\ \rho_{3,w}^{25},\allowbreak\ \rho_{5,x}^{15},\allowbreak\ \rho_{5,x}^{45},\allowbreak\ \rho_{3,7}^{35}$ & $[a_{7,5},b_{19}]\quad\mathsf{C}$ \\
\end{longtable}\endgroup

\end{document}